\documentclass[12pt,reqno]{amsart}
\usepackage[top=1in, bottom=1in, left=1in, right=1in]{geometry}
\usepackage{times, amsthm, amssymb, amsmath, amsfonts, amsthm,mathrsfs,amsrefs,mathtools,bm}
\usepackage{hyperref}
\usepackage[usenames,dvipsnames]{color}
\usepackage{enumitem}
\usepackage{tikz}
\usepackage{tikz-cd}
\usepackage{tikz-3dplot}
\usepackage{subcaption}
\usepackage{caption}
\usepackage{xcolor}
\usepackage[capitalise]{cleveref}
\usepackage{adjustbox}
\usepackage[T1]{fontenc}
\usepackage{booktabs}
\usepackage{tabularx}
\newtheorem{theorem}{Theorem}[section]
\newtheorem{lemma}[theorem]{Lemma}

\newtheorem{corollary}[theorem]{Corollary}

\theoremstyle{definition}
\newtheorem{definition}[theorem]{Definition}
\newtheorem{example}[theorem]{Example}

\newcommand{\<}{\langle}
\renewcommand{\>}{\rangle}
\newcommand{\minus}{\smallsetminus}

\DeclareMathOperator{\spec}{Spec}
\newcommand{\character}{\rho} 
\newcommand{\cellvar}{\zeta}

\newcommand{\N}{\mathbb{N}} 
\newcommand{\Z}{\mathbb{Z}} 
\newcommand{\Q}{\mathbb{Q}} 
\newcommand{\R}{\mathbb{R}} 
\newcommand{\field}{\Bbbk} 
\newcommand{\affinesemigp}{Q} 
\newcommand{\genset}{A} 
\newcommand{\realize}{\R_{\geq0}} 
\newcommand{\dimd}{d}
 
\newcommand{\facF}{F} 
\newcommand{\facG}{G}
\newcommand{\facH}{H} 
\newcommand{\facv}{v} 

\newcommand{\virtual}[1]{\widehat{#1}} 

\DeclareMathOperator{\relint}{RelInt}

\newcommand{\closure}[1]{\overline{#1}} 

\newcommand{\polyhedron}{P} 
\newcommand{\polytope}{\mathcal{P}} 
\newcommand{\crosssection}{K} 
\newcommand{\monoideal}{\mathcal{I}}

\newcommand{\idI}{J}
\newcommand{\ridI}{\idI_{\subcpx}}

\newcommand{\maxid}{\mathfrak{m}} 
\newcommand{\injmodule}{E}
\newcommand{\module}{M} 
\newcommand{\localcoho}[3]{H_{#1}^{#2}(#3)} 
\newcommand{\setS}{S}

\newcommand{\smallk}{k} 
\newcommand{\smalli}{i}

\newcommand{\hplane}{\mathscr{H}} 

\newcommand{\veca}{\mathbf{a}}
\newcommand{\vecb}{\mathbf{b}} 
\newcommand{\vecc}{\mathbf{c}}

\newcommand{\vecf}{\mathbf{f}}

\newcommand{\vecu}{\mathbf{u}} 
\newcommand{\vecv}{\mathbf{v}}

\newcommand{\varx}{x}

\newcommand{\ishida}{L} 
\newcommand{\rchaincpx}[1]{\tilde{\mathcal{C}}\left(#1\right)}

\newcommand{\facelattice}{\mathcal{F}} 
\newcommand{\incidence}{\epsilon} 
\newcommand{\hilbt}{\mathbf{t}} 
\newcommand{\hilb}[2]{\operatorname{Hilb}(#1,#2)} 
 
\newcommand{\tensor}[1]{
  \mathbin{\mathop{\otimes}\displaylimits_{#1}}
}  

\newcommand{\subcpx}{\Delta}

\newcommand{\arngmt}{\mathcal{A}} 
\newcommand{\regionr}[1]{\mathfrak{r}(#1)} 
\newcommand{\regionR}[1]{\mathfrak{R}(#1)} 
\newcommand{\regr}{\mathfrak{r}} 
\newcommand{\regR}{\mathfrak{R}}

\newcommand{\catname}[1]{{\normalfont\textbf{#1}}}
\DeclareMathOperator{\ass}{Ass} 
\DeclareMathOperator{\Hom}{Hom} 
 
\DeclareMathOperator{\nat}{nat}

\DeclareMathOperator{\face}{face}

\usepackage{scalerel}
\usepackage{graphicx}

\colorlet{ignored}{gray!50}
\colorlet{stdmcolor}{black}
\colorlet{idealcolor}{blue!60}
\colorlet{idealregioncolor}{gray!20}

\makeatletter
\@namedef{subjclassname@2020}{
  \textup{2020} Mathematics Subject Classification}
\makeatother

\begin{document}

\title{A Generalization of the Ishida Complex with applications}

\author[Laura Felicia Matusevich]{Laura Felicia Matusevich}
\address[Laura Felicia Matusevich]{Department of Mathematics \\
Texas A\&M University \\ College Station, TX 77843.}
\email[Laura Felicia Matusevich]{matusevich@tamu.edu}
\author[Erika Ordog]{Erika Ordog}
\address[Erika Ordog]{Department of Mathematics \\
Texas A\&M University \\ College Station, TX 77843.}
\email[Erika Ordog]{erika.ordog@tamu.edu}
\author[Byeongsu Yu]{Byeongsu Yu}
\address[Byeongsu Yu]{Department of Mathematics \\
Texas A\&M University \\ College Station, TX 77843.}
\email[Byeongsu Yu]{byeongsu.yu@tamu.edu}

\date{}

\begin{abstract}
We construct a generalize Ishida complex to compute the local cohomology with monomial support of modules over quotients of polynomial rings by cellular binomial ideals. As a consequence, we obtain a combinatorial criterion to determine when such a quotient is Cohen--Macaulay. In particular, this gives a Cohen--Macaulayness criterion for lattice ideals. We also prove a result relating the local cohomology with radical monomial ideal support of an affine semigroup ring to the local cohomology with maximal ideal support of the quotient of the affine semigroup ring by the radical monomial ideal. This requires a combinatorial assumption on the semigroup, which holds for (not necessarily normal) semigroups whose cone is the cone over a simplex.
\end{abstract}

  

\maketitle

\section{Introduction}
\label{sec:intro}

Stanley--Reisner rings and affine semigroup rings are a staple of combinatorial commutative algebra. The former are defined by squarefree monomial ideals, while the latter are defined by toric ideals. Both of these special classes of ideals fall under the general umbrella of binomial ideals, that is, ideals defined by polynomials with at most two terms.

Local cohomology is an important tool from homological commutative algebra, which has proved very useful in combinatorial settings. In this article, we study local cohomology of more general binomial ideals, known as \emph{cellular binomial ideals}. To do this, we generalize and adapt the Ishida complex, which was developed to compute local cohomology for modules over affine semigroup rings with support at the maximal monomial ideal, in two ways: by extending the base ring, and also the supporting monomial ideal. 

Extending the base ring is necessary, as quotients by cellular binomial ideals are not modules over affine semigroup rings. The proof that our generalized Ishida complex indeed computes local cohomology follows along the usual lines of checking that it works at homological degree zero, and then proving vanishing for injectives. The combinatorics involved becomes more challenging in the more general context.

Our original motivation for these developments was was twofold. First, we wanted to have a combinatorial criterion for when a lattice ideal is Cohen--Macaulay. Second, we wanted to better understand a duality result for local cohomology in the Stainley--Reisner case.

\subsection{Lattice ideals}
A \emph{lattice} is a subgroup of $\Z^n$. Given a lattice $L$, its \emph{lattice ideal} is
\[
I_L := \langle x^u-x^v \mid u-v \in L \rangle \subset \field[x_1,\dots,x_n],
\]
where $\field$ is a field.
The \emph{saturation} of a lattice $L$ is $L_{\text{sat}}:=(\Q\oplus_\Z L) \cap \Z^n$. $L$ is \emph{saturated} if it equals its saturation. Lattice ideals corresponding to saturated lattices are known as \emph{toric ideals}, and are easily seen to be prime. The toric ideal $I_{L_{\text{sat}}}$ is known to be a minimal prime of $I_L$, and if $\field$ is algebraically closed, all associated primes of $I_L$ are isomorphic to $I_{L_{\text{sat}}}$ by rescaling the variables~\cite{MR1394747}. 

Quotients by toric ideals are \emph{affine semigroup rings}. Cohen--Macaulayness of affine semigroup rings is well studied, see for instance~\cites{MR304376,MR857437,MY2022}. In the case of general lattice ideals, one can compute Betti numbers using suitable simplicial complexes. In special cases~\cites{MR1649322,MR1475887}, these simplicial complexes have tractable enough homology to provide combinatorial criteria to determine whether a quotient by a lattice ideal is Cohen--Macaulay. Moreover, there is no clear relationship between Cohen--Macaulayness of $I_L$ and $I_{L_{\text{sat}}}$~\cite{MR3957112}.

Local cohomology has been very effective to provide such combinatorial criteria in other situations, but was not previously studied for lattice ideals. This is mainly for two reasons: the lack of specialized tools (which are available for toric ideals but not in general), and the fact that the natural grading group for $I_L$ is $\Z^n/L$ which has torsion.

In this case we tackle both of these issues. First, we generalize the \emph{Ishida complex}, which computes local cohomology for modules over affine semigroup rings to the more general context of lattice ideals. (Actually, we can deal with more general binomial ideals known as \emph{cellular binomial ideals}.) Then we study torsion gradings. The end result is the desired combinatorial Cohen--Macaulayness criterion in the cellular binomial case.

\subsection{Duality for local cohomology}
If $\Delta$ is a simplicial complex, $I_\Delta$ is the associated squarefree monomial ideal, and $\field[\Delta] = \field[x_1,\dots,n]/I_{\Delta}$ the corresponding Stanley--Reisner ring, a comparison of Hilbert series formulas for local cohomology due to Hochster (for maximal ideal support) and Terai (for radical monomial ideal support)~\cite{Terai99} yields the following result (see also~\cite{Huneke07}*{Theorem 6.8} and the references therein)
\begin{equation}\label{eqn:duality}
H_{I_\Delta}^J(\field[x]) = 0 \Longleftrightarrow H^{n-j}_{\maxid}(\field[\Delta]) = 0.
\end{equation}
We asked the question whether there was a relationship in the nonvanishing case, and whether this would hold in the more general context of radical monomial ideals in affine semigroup rings.

This brought us to our second direction of generalization for the Ishida complex, to deal with monomial supports other than the homogeneous maximal ideal. 

In the end, we can prove an isomorphism of local cohomology generalizing~\eqref{eqn:duality} (Theorem~\ref{thm:duality_of_local_cohomologies}), but it requires a combinatorial condition on the affine semigroup ring. Nevertheless, our result does hold for affine semigroup rings whose corresponding cone is the cone over a simplex. We remark that we do not require a normality assumption.

\subsection*{Outline}
In section 2, we recall known results on binomial ideals. In Section 3, we introduce the generalized Ishida complex, which allows us to compute the local cohomology of a module over a polynomial ring quotient by a cellular binomial ideal with radical monomial ideal support. In section 4, we study local cohomology of cellular binomial ideals, and provide a Cohen--Macaulayness criterion. In Section 5, we focus on the local cohomology of a quotient of a simplicial affine semigroup ring by a radical monomial ideal. In section 6, we present our duality result for local cohomology.

\subsection*{Acknowledgments}
We are grateful to Aida Maraj, Aleksandra Sobieska, Catherine Yan, Jaeho Shin, Jennifer Kenkel, Jonathan Monta{\~n}o, Joseph Gubeladze, Kenny Easwaran, Melvin Hochster, Sarah Witherspoon, Semin Yoo, Serkan Ho{\c s}ten, Yupeng Li for inspiring conversations we had while working
on this project.
EO was supported by an NSF Mathematical Sciences Postdoctoral Research Fellowship under award DMS-2103253.

\section{Preliminaries}
\label{sec:preliminary}

\subsection{Affine semigroup rings}
\label{subsec:affine_semigroup}

An \emph{affine semigroup} $\affinesemigp$ is a finitely generated submonoid of $\Z^d$. Throughout this article, we denote by $\genset$ a $d\times n$ integer matrix of rank $d$ whose columns generate $\affinesemigp$, so that $\affinesemigp = \N \genset$.
We assume none of the columns of $\genset$ is the zero vector. 
If $\field$ is a field and $\affinesemigp$ is the affine semigroup given by $\genset$, we denote by $\field[\affinesemigp]=\field[\N \genset]$ the corresponding \emph{affine semigroup ring}. Since $\genset$ has rank $d$, $\field[\affinesemigp]$ is a $d$-dimensional $\field$-algebra.

 The \emph{cone} over the affine semigroup $\affinesemigp$ (or over $\genset$) is the (rational, polyhedral) cone
 $\realize\affinesemigp( = \realize \genset)$, that is, the set of all non-negative real combinations of columns of $A$. This cone is \emph{pointed} if it contains no lines. In this case, the affine semigroup $\affinesemigp$ is also called pointed. A \emph{face} of $\realize \genset$ is a subset of this cone where some linear functional on $\R^{d}$ is maximized over $\realize \genset$. We denote by $\facelattice(\realize\affinesemigp)$ the collection of all faces of $\realize\affinesemigp$. This set forms a lattice under inclusion. For ease in the notation, we identify a face of $\affinesemigp$ with the set of columns of $\genset$ that lie on that face. For a face $\facF \in \facelattice(\affinesemigp)$, the \emph{relative interior} $\relint(\N\facF)$ of the semigroup over $\facF$ is the set of elements of $\N \facF$ that do not belong to any subsemigroups arising from proper faces of $\N\facF$. A \emph{transverse section} $\crosssection$ of $\realize\affinesemigp$ is the intersection of $\realize\affinesemigp$ with a hyperplane which meets all unbounded faces of $\realize\affinesemigp$~\cite{Ziegler95}*{Exercise 2.19}. It is well-known that $\facelattice(\crosssection)$ is canonically bijective to $\facelattice(\realize\affinesemigp)$ as a poset.

Let $\spec_{\text{Mon}}\field[\affinesemigp]$ be the set of all monomial prime ideals of the affine semigroup ring $\field[\affinesemigp]$, ordered by inclusion. There is an order-reversing isomorphism between $\spec_{\text{Mon}}\field[\affinesemigp]$ and the face lattice $\facelattice(\crosssection)$ of the transverse section $\crosssection$ of $\realize\affinesemigp$ given by identifying a face of $\crosssection$ with the (prime) ideal generated by all monomials whose exponents do not belong to the corresponding face of $\realize\affinesemigp$. 

It is known that the set of all radical ideals is in one to one correspondence with the set of all polyhedral subcomplexes of $\facelattice(\affinesemigp)$. 

\subsection{Binomial ideals}
\label{subsec:affinesemigp}

Let $\field[x]=\field[x_1,x_2,\dots,x_d]$ be the polynomial ring over a field $\field$. A \emph{binomial} is a polynomial having at most two terms. A \emph{binomial ideal} is an ideal generated by binomials. We are interested in three types of binomial ideals, \emph{lattice ideals}, \emph{toric ideals}, and \emph{cellular binomial ideals}.

Let $L_{\character}$ be a subgroup of $\Z^{d}$. A \emph{partial character} $\character: L_{\character} \to \field^{\ast}$ on $\Z^{d}$ is a group homomorphism, where $\field^{\ast}$ is the multiplicative group of $\field$. The \emph{lattice ideal} $I(\character)$ corresponding to $\character$ is the ideal in $\field[x]$ defined as
\begin{equation}
\label{eqn:deflattice}
I(\character) :=\<x^{\vecu} - \rho(\vecu-\vecv)x^{v} \mid \vecu-\vecv \in L_{\character} \> \subset \field[x_1,\dots,x_n].
\end{equation}
We remark that in~\cite{MR1394747}, these lattice ideals are denoted by
$I(\character)_{+}$, while $I(\character)$ is used for lattice ideals in Laurent polynomial rings. Since we do not need the more general context, we use~\eqref{eqn:deflattice} for economy in the notation.

The \emph{saturation} $(L_{\rho})_{\text{sat}}$ of a lattice $L_{\rho}$ is $(L_{\rho})_{\text{sat}}:=(\Q \otimes_{\Z}L_{\rho})\cap \Z^{d}$. 
A lattice is \emph{saturated} if $L_{\rho}=(L_{\rho})_{\text{sat}}$. 
If $\field$ is algebraically closed, then $I(\character)$ is prime if and only if $L_{\rho}$ is saturated~\cite{MR1394747}*{Theorem 2.1.c.}.
Furthermore, the associated primes of a lattice ideal are lattice ideals corresponding to the saturation of the underlying lattice~\cite{MR1394747}*{Corollary 2.5}.  
In particular,~\cite{MR1394747}*{Corollary 2.5} implies that if $\field$ is an algebraically closed field of characteristic zero, then lattice ideals are radical, and  primary lattice ideals are prime. 

An ideal $I \subset \field[x]$ is \emph{cellular} if all variables are either nonzero divisors modulo $I$ or nilpotent modulo $I$. The nonzero divisor variables are known as the \emph{cellular variables} of $I$. Let $\cellvar \subset [n]$ be the set of all indices of cellular variables of a cellular binomial ideal $I$, then $I$ is \emph{$\cellvar$-cellular}. Lattice ideal are special cases of cellular binomial ideals, where all variables are cellular. Also, $I \cap \field[\N^{\cellvar}]$ is a lattice ideal~\cite{MR3556446}. 
The following result can be found in~\cite{MR1394747}*{Section~6} and also in~\cite{MR3556446}*{Corollary 3.5}.
\begin{theorem}
\label{thm:associated_primes_of_cellular}
Let $I$ be a $\cellvar$-cellular binomial ideal in $\field[x]$. The associated primes of $I$ are the ideals $\field[x] \cdot P+\< x_{i} \mid i \in \cellvar^{c}\>$, where $P \subset \field\left[\N^{\cellvar}\right]$ runs over the associated primes of lattice ideals of the form $(I: m) \cap \field\left[\N^{\cellvar}\right]$, for monomials $m \in \field\left[\N^{\cellvar^{c}}\right]$.
\end{theorem}

A \emph{toric ideal} is the prime lattice ideal $I(\chi)$ corresponding to the trivial character $\chi:L_{\chi} \to \{1\} \subset \field^{\ast}$. An \emph{affine semigroup} $\affinesemigp$ is the quotient of $\N^{d}$ by the equivalence relation $\sim_{L_{\chi}}$ given by $\vecu \sim_{L_{\chi}} \vecv \iff \vecu-\vecv \in L_{\chi}$. A quotient of polynomial ring by corresponding prime lattice ideal $\field[x]/I(\chi)$ is isomorphic to an affine semigroup ring $\field[\affinesemigp]$~\cite{CCA}*{Theorem 7.3}.

From now on, we assume that the base field $\field$ is algebraically closed.

\section{Generalized Ishida complex}
\label{sec:gen_Ishida}

Given a lattice ideal $I$ (resp. $\cellvar$-cellular binomial ideal $I$), pick a minimal associated prime ideal $J$ of $I$ (resp. of $I \cap \field[\N^{\cellvar}]$). Since $\field$ is algebraically closed, $J$ is a binomial prime ideal, and the quotient $\field[\N^{\cellvar}]/J$ is isomorphic (by rescaling the variables) to an affine semigroup ring $\field[\affinesemigp]$ with $\affinesemigp=\N \genset$. The natural projection map $\field[x]/I \to \field[x]/J \cong \field[\affinesemigp]$ (resp. $\field[x]/I \to \field[x]/(J+\<x_{i}: i \in \cellvar^{c}\>) \cong \field[Q]$) induces $\genset$-\emph{grading} on $\field[x]/I$; in other words, for any monomial $\overline{x^{\vecu}} \in \field[x]/I$, the \emph{$\genset$-degree} of $\overline{x^{\vecu}}$ is $\deg_{\genset}(\overline{x^{\vecu}})=\genset \cdot \vecu$ (resp. $\genset \cdot \vecu^{\cellvar}$, where $\vecu^{\cellvar}:= (u_{i} \in \vecu; i \in \cellvar)$). 
We specify the $A$-grading on $\field[x]/I$ using the triple $(I,J,A)$ unless the minimal prime $J$ or the generators of affine semigroup $\N A$ are understood in context.

In this section, we generalize the Ishida complex to compute the local cohomology of a lattice ideal (resp. $\cellvar$-cellular binomial ideal) $(I,J,\genset)$ supported at a contraction of a radical monomial ideal of $\field[\N\genset]=\field[\affinesemigp]$. Recall that the contraction of a radical monomial ideal is also a monomial ideal in $\field[x]/I$.

\subsection{Ishida complex}
\label{subsec:Ishida}

The Ishida complex was originally developed for computing the local cohomology of modules over pointed affine semigroup rings supported on the graded maximal ideal~\cite{MR977758}. Given a pointed affine semigroup $\affinesemigp$ with transverse section $\crosssection$, there is a canonical isomorphism $\virtual{-}:\facelattice(\crosssection) \to \facelattice(\affinesemigp)$ given as follows: $\virtual{\facF}$ is the minimal face of $\affinesemigp$ such that $\realize\virtual{\facF} \supseteq \realize\facF$. Since $\affinesemigp$ is pointed, it has a unique zero-dimensional face, namely the origin, which corresponds to the (-1)-dimensional face $\varnothing$ of $\crosssection$. As a CW complex, $\crosssection$ has an incidence function $\incidence: \bigoplus_{\smalli=-1}^{\dimd-1}\facelattice(\crosssection) ^{\smalli}\times  \facelattice(\crosssection)^{\smalli+1}   \to \{0,\pm1 \}$ that has a nonzero value when two faces are incident.
We now recall the definition of the Ishida complex~\cite{MR977758}.

\begin{definition}
\label{def:ishida_cpx}
Let $\maxid$ be the maximal monomial ideal of $\field[\affinesemigp]$. The set of all $k$-dimensional faces in $\facelattice(\crosssection)$ is denoted by $\facelattice(\crosssection)^{k}$. Let $\ishida^{\bullet}$ be the chain complex 
\[
\begin{tikzcd}
\ishida^{\bullet}: 0\arrow[r] & \ishida^{0} \arrow[r,"\partial"] &\ishida^{1} \arrow[r,"\partial"] & \cdots \arrow[r,"\partial"] & \ishida^{d} \arrow[r,"\partial"] & 0, \qquad L^{\smallk} := \bigoplus\limits_{\facF \in \facelattice(\crosssection)^{\smallk-1}}\field[\affinesemigp-\N\virtual{\facF}]
\end{tikzcd}
\]
where the differential $\partial :L^{\smallk} \to L^{\smallk+1}$ is induced by the componentwise map $\partial_{\facF,\facG}$ with $\facF \in  \facelattice(\crosssection)^{\smallk-1}$, $\facG \in  \facelattice(\crosssection)^{\smallk}$ such that
\[
\partial_{\facF,\facG}: \field[\affinesemigp-\N\virtual{\facF}] \to \field[\affinesemigp-\N\virtual{\facG}] \text{ to be } \begin{cases} 0 & \text{ if } \facF \not\subset \facG \\ \incidence(\facF,\facG)\cdot\nat & \text{ if } \facF \subset \facG  \end{cases}
\]
with $\nat$, the canonical injection $\field[\affinesemigp-\N\virtual{\facF}] \to \field[\affinesemigp-\N\virtual{\facG}]$ when $\facF \subseteq \facG$. We say that $\ishida^{\bullet} \otimes_{\field[\affinesemigp]}\module$ is the \emph{Ishida complex of a $\field[\affinesemigp]$-module $\module$ supported at the maximal monomial ideal}.
\end{definition}

The cohomology of the Ishida complex of $\module$ supported at the maximal monomial ideal is isomorphic to the local cohomology of $\module$ supported at the maximal monomial ideal.

\begin{theorem}[\cite{MR977758}*{Theorem 6.2.5}]
\label{thm:Ishida}
For any $\field[\affinesemigp]$-module $\module$, and all $\smallk \geq 0$,
\[
\localcoho{\maxid}{\smallk}{\module} \cong H^{\smallk}(\ishida^{\bullet} \tensor{\field[\affinesemigp]}\module).
\]
\end{theorem}

Let $\ridI$ be the radical monomial ideal of $\field[\affinesemigp]$ associated to a subcomplex $\Delta \subset \facelattice(\affinesemigp)$. Then, $\sqrt{\ridI \cdot \field[x]/I}$ denotes the contraction of $\ridI$ via $\field[x]/I \to \field[x]/J \cong \field[\affinesemigp]$. To compute the local cohomology of a $\field[x]/I$-module supported on $\sqrt{\ridI \cdot \field[x]/I}$, we construct a (generalized) Ishida complex below.

Let $\crosssection_{\ridI}$ be a transverse section of the polyhedron $\realize\{u \in \Z^{\cellvar}\mid x^{u} \in \ridI\}$ with the canonical isomorphism $\virtual{-} : \facelattice(\crosssection_{\ridI}) \to \facelattice(\affinesemigp)$ where $\virtual{\facF}$ is the minimal face of $\affinesemigp$ such that $\realize\virtual{\facF} \supseteq \realize\facF$. The set of all $k$-dimensional faces in $\facelattice(\crosssection_{\ridI})$ is denoted by $\facelattice(\crosssection_{\ridI})^{k}$. Also, $(\field[x]/I)_{\virtual{\facF}}$ refers to the localization of $\field[x]/I$ by the multiplicative set consisting of all monomials in $\field[\N^{\cellvar}]$ whose $A$-graded degrees are in $\N\virtual{\facF}$. 

\begin{definition}[Generalized Ishida complex]
\label{def:gen_ishida}
Let $\ishida^{\bullet}$ be the chain complex 
\[
\begin{tikzcd}
\ishida^{\bullet}: 0\arrow[r] & \ishida^{0} \arrow[r,"\partial"] &\ishida^{1} \arrow[r,"\partial"] & \cdots \arrow[r,"\partial"] & \ishida^{d} \arrow[r,"\partial"] & 0, \quad L^{\smallk} := \bigoplus\limits_{\facF \in \facelattice(\crosssection_{\ridI})^{\smallk-1}}\left(\field[x]/I\right)_{\virtual{\facF}}
\end{tikzcd}
\]
where the differential $\partial :L^{\smallk} \to L^{\smallk+1}$ is induced by a componentwise map $\partial_{\facF,\facG}$ with two faces $\facF \in  \facelattice(\crosssection_{\ridI})^{\smallk-1}$, $\facG \in  \facelattice(\crosssection_{\ridI})^{\smallk}$ such that
\[
\partial_{\facF,\facG}: \left(\field[x]/I\right)_{\virtual{\facF}} \to\left(\field[x]/I\right)_{\virtual{\facG}} \text{ to be } \begin{cases} 0 & \text{ if } \facF \not\subset \facG \\ \incidence(\facF,\facG)\cdot\nat & \text{ if } \facF \subset \facG  \end{cases}
\]
with $\nat$, the canonical injection $\left(\field[x]/I\right)_{\virtual{\facF}} \to \left(\field[x]/I\right)_{\virtual{\facG}}$ when $\facF \subseteq \facG$. We say that $\ishida^{\bullet} \otimes_{\field[x]/I}\module$ is the \emph{Ishida complex of a $(\field[x]/I)$-module $\module$ supported at the radical monomial ideal $\sqrt{\ridI \cdot \field[x]/I}$}.
\end{definition}

The following theorem is the main result in this section. The proof is adapted from~\cite{BH_CMrings}. The key ingredients are~\cref{lem:first_step_ishida} and~\cref{lem:second_step_ishida} which are given later.

\begin{theorem}
\label{thm:Ishida_gen}
For any $\field[x]/I$-module $\module$, and all $\smallk \geq 0$,
\[
\localcoho{\ridI \cdot \field[x]/I}{\smallk}{\module} \cong \localcoho{\sqrt{\ridI \cdot \field[x]/I}}{\smallk}{\module} \cong H^{\smallk}(\ishida^{\bullet} \tensor{\field[x]/I}\module).
\]
\end{theorem}

The first step in our proof is to verify that the zeroth homology of the generalized Ishida complex computes torsion.

\begin{proof}
This follows from~\cref{lem:first_step_ishida} and~\cref{lem:second_step_ishida} together with the fact that all the summands of the components of the Ishida complex are flat, thus $-\tensor{\field[x]/I} \ishida^{\bullet}$ is an exact functor.
\end{proof}

\begin{lemma}
\label{lem:first_step_ishida}
$H^{0}(\ishida \otimes_{\field[x]/I}M) \cong (0:_{M} \ridI^{\infty})$.
\end{lemma}
\begin{proof}
It suffices to show that 
\[
\sqrt{\left\<\bigcup\limits_{\facF \in \facelattice(\crosssection_{\ridI})^{0}} \{ \overline{\varx^{\vecu}} \in \field[\N^{\cellvar}]/(I\cap \field[\N^{\cellvar}]): \deg_{A}(\vecu) \in \relint(\N\virtual{\facF})\} \right\>} = \sqrt{\ridI \cdot \field[x]/I}.
\]
The generators of the left hand-side ideal might be not the same as those of the multiplicative sets inducing localization of components in $L^{1}$ when $I$ is not toric. However, they admit the same radical ideal in $\field[x]/I$.

First, let $\overline{\varx^{\vecu}} \in \field[x]/ I$ be an element whose $A$-degree is in $\N\virtual{\facF}$ for a vertex $\facF$ of $\crosssection_{\ridI}$. 
The canonical map $\field[x]/I \to \field[x]/J \cong\field[\N\genset]$ sends $\overline{\varx^{\vecu}}$ to $\varx^{\deg_{A}(\vecu)} \in \field[\N\genset]$ where $\deg_{A}(\vecu) \in \relint(\N\virtual{\facF})$. 
Since $\virtual{\facF} \not\in \Delta$, $\varx^{\deg_{A}(\vecu)} \in \ridI$, we have that $\overline{\varx^{\vecu}} \in \sqrt{\ridI \cdot \field[x]/I}$ by the correspondence between polyhedral subcomplexes and the radical monomial ideals.

Conversely, let $f \in \field[\N^{\cellvar}]$ be a preimage of a monomial in $\ridI \subset \field[\N\genset]$ and $g \in \field[x]/I$ be an $A$-homogeneous element of $\field[x]/I$. Then, $f = \overline{x^{\vecu}}$ for some $\vecu \in \N^{\cellvar}$ such that $A \vecu \in \N\overline{\facF}$ for some $\facF \in \facelattice(\crosssection_{\ridI})$. If $\dim\facF =0$, $fg$ is in the left hand-side of the equation. Suppose $\dim \facF >0$; then $\facF$ has vertices $\{v_{1},\dots, v_{m}\}$. Then, $\vecu$ is a linear combination of elements of $\N\virtual{v_{i}}$ over $\Q$, say $\vecu= c_{1}\vecu_{1}+ \cdots + c_{m}\vecu_{m}$ where $\vecu_{i} \in \N\virtual{v_{i}}$ and $c_{i} \in \Q$. Multiplying $\vecu$ by a suitable number $N$, we may assume that $c_{i} \in \N$. Then, $f^{N} = \prod_{i=1}^{m}(f_{i})^{c_i}$ where $f_{i}= \overline{x^{\vecu_{i}}}$, which implies that $f^{N}g^{N} $ is in the left hand-side, thus $fg$ is in the left hand-side.
\end{proof}

To complete the proof of our main result, we need to check that the generalized Ishida complex is exact on injectives. This is stated in the following lemma, which requires three auxiliary results.

\begin{lemma}
\label{lem:second_step_ishida}
If $M$ is an injective $\field[x]/I$-module, then $\ishida^{\bullet}\tensor{\field[x]/I} M$ is exact.
\end{lemma}
\begin{proof}
It suffices to check the case when $M$ is an injective indecomposable module $E(\field[x]/P)$ over a prime ideal $P$ containing $I$. Let $\genset_{i}$ be the $i$-th column of $\genset$. 
The set $\facF:=\{A_{i}\mid x_{i}E(\field[x]/P) \cong E(\field[x]/P)\}$ is called the \emph{face corresponding to $P$}. ~\cref{lem:injective_module_with_face} shows that this is indeed a face of $\N\genset$. \cref{lem:tensoring_injective_module} shows that $\ishida \tensor{\field[x]/I}E(\field[x]/P)$ is exact using \cref{lem:closed_subpolytope}.

\end{proof}

We start verifying our proposed face is a face.

\begin{lemma}
\label{lem:injective_module_with_face}
Given a monomial prime ideal $P$ containing $I$, $\facF':=\{A_{i}\mid x_{i}E(\field[x]/P) \cong E(\field[x]/P)\}$ is a face of $\N\genset$.
\end{lemma}
\begin{proof}
Since $\ass(E(\field[x]/P)) =\{P\}$ and the module is indecomposable, $x_{i}E(\field[x]/P)$ is either $0$ if $x_{i} \in P$ or $E(\field[x]/P)$ if $x_{i}\not\in P$. Now suppose that the given set $\facF'$ is not a face; then there exists a minimal face $\facF$ whose relative interior intersects with the relative interior of $\facF'$. Pick $A_{j} \in \facF \setminus \facF'$. Then the corresponding variable $x_{j}$ induces the zero map on $E(\field[x]/P)$. If $A_{j}$ is not in the relative interior of $\facF$, let $\vecf \in \relint(\N\facF)$ so that $\vecf= \sum_{A_{i} \in \facF'} c_{i}A_{i}$ for some $c_{i} \in \N$. 
Choose a suitable $N_{1} \in \N$ such that $N_{1}\vecf=\sum_{A_{i} \in \facF'} c'_{i}A_{i} + dA_{j}$ for some nonnegative $c'_{i}\in \Q$  and $d \in \Q_{>0}$. By construction, $N_{1}\vecf$ is in the lattice $(L_{\rho},\rho)$. 
Thus, there exists $N_{2}>N_{1}$ such that $N_2\vecf = \sum_{A_{i} \in \facF'} d_{i}A_{i} + d'A_{j}$ where $d_{i},d\in \N$, $d>0$. 
But then $x^{N_{2}\vecf} E(\field[x]/P) = 0$, which is a contradiction. 
If $A_{j}$ is in the relative interior of $\facF$, a similar argument gives another contradiction. 
\end{proof}

The following is necessary to prove exactness.

\begin{lemma}
\label{lem:closed_subpolytope}
Given a face $\facF \in \facelattice(\N\genset)$, $\crosssection_{\ridI}^{\cap\facF}:= \crosssection_{\ridI} \cap \realize\facF$ is a face of $\crosssection_{\ridI}$.
\end{lemma}
\begin{proof}
If $\facF=\N\genset$, then the statement is clear. 
Assume $\dim\facF < \dim\N\genset$. 
First, we claim that for any $\facG \in \facelattice(\crosssection_{\ridI})$, $\facG \subseteq \realize\facF$ if and only if $\virtual{\facG} \subseteq \facF$. One direction follows straight from the definition of $\virtual{\facG}$. Conversely, assume that $\virtual{\facG} \not\subseteq \facF$. Then, $\relint(\realize\virtual{\facG}) \cap \realize\facF = \emptyset$ implies $\relint(\facG)  \cap \realize\facF = \emptyset$. Therefore $\facG \not\subseteq \realize\facF$. This claim shows that $\crosssection_{\ridI}^{\cap\facF}$ is the union of all faces $\facG \in \facelattice(\crosssection_{\ridI})$ such that $\virtual{\facG} \subseteq \facF$. Thus $\crosssection_{\ridI}^{\cap\facF}$ can be regarded as a realized subcomplex of $\facelattice(\crosssection_{\ridI}^{\cap\facF}):=\{\facG \in \facelattice(\crosssection_{\ridI}): \virtual{\facG} \subseteq \facF\}$.

Next, we claim that $\facelattice(\crosssection_{\ridI}^{\cap\facF})$ has a unique maximal element. Suppose not; let $\facG_{1}$ and $\facG_{2}$ be two distinct faces of $\facelattice(\crosssection_{\ridI}^{\cap\facF})$ of maximal dimension. Then, $\facF \supseteq \virtual{\facG_{1}} + \virtual{\facG_{2}}$ implies that there is a face $\virtual{\facG_{1}} \vee \virtual{\facG_{2}} \in \facelattice(\N\genset)$ such that $\facF \supseteq \virtual{\facG_{1}} \vee \virtual{\facG_{2}}$, the join of the two faces. 
Since $\facG_{1}$ and $\facG_{2}$ are distinct and the same dimension, $\facG_{1} \vee\facG_{2} \neq \facG_{1}$ or $\facG_{2}$. Therefore $\facG_{1} \vee\facG_{2} \supsetneq \facG_{1} \cup \facG_{2}$. Hence, $\realize(\virtual{\facG_{1}} \vee \virtual{\facG_{2}}) \supseteq \realize(\virtual{\facG_{1} \vee \facG_{2}})$, which implies $\facF \supseteq \virtual{\facG_{1}} \vee \virtual{\facG_{2}} \supseteq \virtual{\facG_{1} \vee \facG_{2}}$. Hence, $\facG_{1} \vee \facG_{2} \subseteq \realize\facF$, therefore $\facG_{1} \vee \facG_{2} \in \facelattice(\crosssection_{\ridI}^{\cap\facF})$, contradicting the maximality of $\facG_{1}$ and $\facG_{2}$. We conclude $\facelattice(\crosssection_{\ridI}^{\cap\facF})$ has a unique maximal element, say $\facH$.

Lastly, we claim that $\facelattice(\facH) = \facelattice(\crosssection_{\ridI}^{\cap\facF})$, which implies $\crosssection_{\ridI}^{\cap\facF} =H$. For any $\facG \in \facelattice(\crosssection_{\ridI}^{\cap\facF})$, let $\facG':= \facG \vee \facH$ in $\facelattice(\crosssection_{\ridI})$. Then, $\emptyset \neq \relint(\virtual{\facG'}) \cap \realize(\virtual{\facG} \cup \virtual{\facH})  \subseteq \relint(\virtual{\facG'}) \cap \realize\facF  \implies \virtual{\facG'} \subseteq \facF \implies \facG' \in \facelattice(\crosssection_{\ridI}^{\cap\facF})$. By the maximality of $\facH$, $\facG' = \facH$, which implies $\facG \subseteq \facH$.
\end{proof}

\begin{lemma}
\label{lem:tensoring_injective_module}
Given a monomial prime ideal $P$ whose corresponding face is $\facF$, for $k \geq 1$,
\[
\ishida^{\smallk} \tensor{\field[x]/I} \injmodule(\field[x]/P)  =  \bigoplus_{\facG \in \facelattice\left(\crosssection_{\ridI}^{\cap\facF}\right)^{\smallk-1}}\injmodule(\field[x]/P) \cong \Hom_{\Z}\left( \rchaincpx{\crosssection_{\ridI}^{\cap\facF}}(-1), \injmodule\left(\field[x]/P\right)\right)
\]
where $\rchaincpx{\crosssection_{\ridI}^{\cap\facF}}$ is the reduced chain complex of $\crosssection_{\ridI}^{\cap\facF}$ as a CW complex.
\end{lemma}

\begin{proof}
\cref{lem:injective_module_with_face} shows that for any $\facF, \facG \in \facelattice(\N\genset)$,
\[
\injmodule(\field[x]/P) \tensor{\field[x]/I} \left(\field[x]/I\right)_{\facG} =  \begin{cases} 0 & \text{ if } \facG  \not\subseteq \facF \\  \injmodule(\field[x]/P)  & \text{ if } \facG \subseteq \facF \end{cases}.
\]
If $\facF$ is not the image of a face in $\crosssection_{\ridI}$, then no sub-face of $\facF$ is the image of a face of $\crosssection_{\ridI}$. 
Otherwise, there is a face $\facG \subseteq \facF$ containing an unbounded face of $\realize\{u \in \Z^{\cellvar} \mid x^{u} \in \ridI\}$.
Then by the correspondence between radical monomial ideals and subcomplexes of $\facelattice(\affinesemigp)$, $\relint(\facF)$ contains an element of an unbounded face of $\realize\{u \in \Z^{\cellvar} \mid x^{u} \in \ridI\}$, a contradiction. Therefore, no images of faces are subsets of $\facF$. This implies that $\crosssection_{\ridI}^{\cap\facF}=0$ and $\ishida^{\smallk} \tensor{\field[\N\genset]} \injmodule(\field[x]/P) =0$ for $\smallk \geq 1$. 

Otherwise, $\facF:= \virtual{\facF'}$ for some $\facF' \in \facelattice(\crosssection_{\ridI})$. Since $\virtual{\facG'} \subseteq \facF$ if and only if $\facG' \in \facelattice(\crosssection_{\ridI}^{\cap\facF})$ by~\cref{lem:closed_subpolytope}, so the first equality holds. 

Now observe that $\Hom_{\Z}( \Z, \injmodule(\field[x]/P))\cong\injmodule(\field[x]/P)$ as a $\field[\N\genset]$-module. 
Thus,
\begin{align*}
\ishida^{\bullet} \tensor{\field[\N\genset]} \injmodule(\field[x]/P) = \bigoplus_{\facG \in \facelattice(\crosssection_{\ridI}^{\cap\facF})^{\bullet-1}}\injmodule(\field[x]/P) \cong \bigoplus_{\facG \in \facelattice(\crosssection_{\ridI}^{\cap\facF})^{\bullet-1}}\Hom_{\Z} ( \Z, \injmodule(\field[x]/P) )  
\\  \cong \Hom_{\Z} \left( \bigoplus_{\facG \in \facelattice(\crosssection_{\ridI}^{\cap\facF})^{\bullet-1}}\Z, \injmodule(\field[x]/P) \right) = \Hom_{\Z} \left( \rchaincpx{\crosssection_{\ridI}^{\cap\facF}}(-1), \injmodule(\field[x]/P) \right).
\end{align*}
\end{proof}

\section{Local cohomology with monomial support for cellular binomial ideals}
\label{sec:Hochster}

In this section we express the Hilbert series of the local cohomology with monomial support of $\field[x]/I$ as a (formal) finite sum of rational functions when $I$ is a lattice ideal (Theorem~\ref{thm:hochster_lattice}) or a cellular binomial ideal (Theorem~\ref{thm:hochster_cellular}).
As a corollary, we provide a generalization of Reisner's criterion to the context of cellular binomial ideals, which gives a Cohen-Macaulay characterization for $\field[x]/I$ in terms of the cohomology of finitely many chain complexes (Corollary~\ref{cor:reisner}).
Let $(I,J,A)$ be a tuple consisting of a lattice ideal (resp. cellular binomial ideal), a minimal prime ideal $J$ of $I \cap \field[\N^{\cellvar}]$, and the corresponding affine semigroup $\N\genset=\affinesemigp$. 
Then $J$ is also a prime lattice ideal and we may assume after rescaling the variables that
$J=I(\xi)$ is toric, with lattice $L_{\xi} = (L_{\rho})_{\text{sat}}$. Let $T:= L_{\xi}/L_{\rho}$ be the corresponding torsion abelian group, then
\[
\Z^{d}/L_{\rho} \cong T \oplus \Z\genset.
\]
We may induce a fine grading of $\field[x]/I$ by $T \oplus \Z\genset$ as follows: for any $\overline{x^{\vecu}} \in \field[x]/I$ for some $\vecu \in \Z^{d}$, $\deg_{T,\genset}(\overline{x^{\vecu}}) := ( \vecu+L_{\rho} ,\genset \cdot \vecu)$. Here we use $\overline{x^{\vecu}}$ to indicate the image of $x^{\vecu} \in \field[x]$ in $\field[x]/I$.

Let $\monoideal$ be a monomial ideal of an affine semigroup ring $\field[\affinesemigp]$. 
Let $\deg(\field[\affinesemigp]/\monoideal):=\{ \veca \in \Z^{d}\mid (\field[\affinesemigp]/\monoideal)_{\veca} \neq 0 \}$. 
A \emph{proper pair} $(\veca,\facF)$ of $I$ is a pair such that $\veca + \N\facF \subset \deg(\field[\affinesemigp]/\monoideal)$. 
Given two pairs $(\veca,\facF)$ and $(\vecb,\facG)$, we say $(\veca,\facF)<(\vecb,\facG)$ if $\veca + \N\facF \subset \vecb + \N\facG$. 
A \emph{degree pair} of $\monoideal$ is a maximal element of the set of all proper pairs with the given order. 
Two pairs $(\veca,\facF)$ and $(\vecb,\facF)$ with the same face $\facF$ \emph{overlap} if the intersection $(\veca+\N\facF) \cap (\vecb+\N\facF)$ is nonempty. 
Overlapping is an equivalence relation on pairs; an \emph{overlap class} $[\veca,\facF]$ is an equivalence class containing the degree pair $(\veca,\facF)$. Let $\bigcup[\veca,\facF]$ be the set of all degrees belonging to $\veca'+\N\facF$ for some degree pair $(\veca',\N\facF)$ in $[\veca,\facF]$. 
The \emph{original degree space} $\bigcup\deg(\field[\affinesemigp]/\monoideal)$ is the union of $\deg((\field[\affinesemigp]/\monoideal)_{P})$ for all monomial prime ideals $P$ of $\field[\affinesemigp]$. 
The \emph{degree pair topology} is the smallest topology on $\bigcup\deg(\field[\affinesemigp]/\monoideal)$ such that for any overlap class $[\veca,\facF]$ of any localization $(\field[\affinesemigp]/\monoideal)_{P}$, the set $\bigcup[\veca,\facF]$ is both open and closed. 
These notions were introduced in~\cites{STV95,STDPAIR,MY2022}.

In this article, we extend the notion of degree space further. 
Suppose that $\arngmt$ is the \emph{hyperplane arrangement} consisting of the supporting hyperplanes of the facets of $\R_{\geq 0}\affinesemigp$. Let $\regionr{\arngmt}$ be the set of regions of $\arngmt$, where a \emph{region} is a connected components of $\R^{d} \setminus \bigcup_{\hplane \in \arngmt}\hplane$~\cite{Stanley07}. Then, for any region $\regr \in \regionr{\arngmt}$, regard $\regr \cap (\Z^{d} \minus \bigcup\deg(\field[\affinesemigp]/\monoideal))$ as a space with the trivial topology. 
The \emph{extended degree space} $\Z\affinesemigp$ of $\monoideal$ is the disjoint union $\Z\affinesemigp= \left(\bigcup\deg(\field[\affinesemigp]/\monoideal)\right) \cup \big( \bigcup_{\regr \in \regionr{\arngmt}}\regr \cap (\Z^{d} \minus \bigcup\deg(\field[\affinesemigp]/\monoideal) \big)$ as a topological space. (As a set, $\Z \affinesemigp$ equals $\Z^d$.)
Lastly, let $\mathcal{G}(\Z\affinesemigp)$ (resp. $\mathcal{G}(\field[\affinesemigp]/\monoideal)$) be the collection of all minimal open sets of the extended (resp. original) degree space of $\monoideal$. 
Since $\mathcal{G}(\field[\affinesemigp]/\monoideal)$ is finite~\cite{MY2022}*{Lemma 4.3} and $\regionr{\arngmt}$ is finite, 
$\mathcal{G}(\Z\affinesemigp)$ is also finite.

Given $t:=\vecu+L_{\rho}$, let $\monoideal_{t}$ be the following ideal in $\field[\affinesemigp]$ 
\[
\monoideal_{t}:=\<x^{\genset \cdot \vecu} \mid \deg_{T,\genset}(\overline{x^{\vecu}}) = (t,A\cdot \vecu) \text{ and } \overline{x^{\vecu}} \in \field[x]/\monoideal \>.
\]
For an open set $O \in \mathcal{G}(\field[x]/\monoideal_{t})$, let $\rchaincpx{O}$ be the graded part of the generalized Ishida complex associated to the element $\overline{x^{\vecu}}$ whose torsion degree is $t$ and whose $A$-degree is $A \vecu \in O$. This is well-defined  regardless of choice of $\overline{x^{\vecu}}$, as is stated below.

\begin{lemma}
\label{lem:grain_chaff}
If $\overline{x^{\vecu}}$ and $\overline{x^{\vecv}}$ with the same torsion degree are in the same minimal open set of the extended degree space of $\monoideal_{t}$, then their corresponding graded parts of the Ishida complex coincide.
\end{lemma}
\begin{proof}
If neither $A \cdot \vecu$ nor $A \cdot \vecv$ are in the original degree space of $\monoideal_{t}$, they must be in the same region of the hyperplane arrangement. Hence, for any localization by a face $\facF$, either both degrees are  $\deg((\monoideal_{t}) \cdot\field[\N\genset -\N\facF])$ or they do not belong to the same localization.

If both are in the original degree space, let $[\veca,\facF]$ be an overlap class whose degree set $\bigcup[\veca,\facF]$ contains $A \cdot \vecu$ and $A \cdot \vecv$, and such that $F$ is minimal with this property. Then $\overline{x^{\vecu}}$ and $\overline{x^{\vecv}}$ do not appear in the localization of $\field[x]/I$ by a multiplicative set generated by variables corresponding to a proper face of $\facF$. Conversely, if there is no overlap class $[\veca,\facG]$ containing $A \cdot \vecv$, this means that $A \cdot \vecv$ appears on every localization of $\affinesemigp$ by faces containing $\facG$. This completely determines the graded parts of the Ishida complex.
\end{proof}

\begin{theorem}
\label{thm:hochster_lattice}
Given a lattice ideal $I$, define $\affinesemigp$ as before. T
he multi-graded Hilbert series for the $i$th local cohomology module of $\field[x]/I$ supported on the inverse image of the radical monomial ideal $\ridI \subset \field[\affinesemigp]$ 
with respect to the  $T \otimes \Z\genset$-grading is
\[
\hilb{\localcoho{\ridI}{i}{\field[x]/I}}{\hilbt}=\sum_{t \in T} \sum_{O \in \mathcal{G}(\field[\N\genset]/\monoideal_{t})}\dim H^{i}(\rchaincpx{O};\field) \sum_{\substack{u \in \Z^{d} \\ Au \in O}}x^{\overline{u}}.
\]
\end{theorem}
\begin{proof}
This is a direct consequence of \cref{lem:grain_chaff}.
\end{proof}
Since $T$ is finite and $\mathcal{G}(\field[x]/\monoideal_{t})$ is finite for all $t \in T$, the sum is finite. Moreover, each minimal open set $O \in \mathcal{G}(\field[x]/\monoideal_{t})$ is the set of lattice points in a convex polyhedron, so that $\sum_{\substack{u \in \Z^{d} \\ Au \in O}}x^{\overline{u}}$ can be written as a rational function~\cites{Barvinok99,Barvinok03}.

For the case of a $\cellvar$-cellular binomial ideal $(I,J,A)$, let $M$ be the multiplicative set consisting of monomials on the nilpotent variables of $\field[x]/I$. Then for each $m \in M$, $(I:m) \cap \field[\N^{\cellvar}]$ is a lattice ideal containing $I\cap\field[\N^{\cellvar}]$. Hence, according to \cref{thm:associated_primes_of_cellular}, we may pick an associated prime ideal $J_{m}$ of $I \cap\field[\N^{\cellvar}]$ whose extension $J_{m} \field[x] +\<x_{i} \mid i \in \cellvar^{c} \>$ is the associated prime containing $J$. Moreover, the quotient $\field[x]/(J_{m} \field[x] +\<x_{i} \mid i \in \cellvar^{c} \>)$ is isomorphic to an affine semigroup ring $\field[\N\genset_{m}]$ for some integer matrix $\genset_{m}$. The canonical projection $\field[x]/J \to \field[x]/(J_{m} \field[x] +\<x_{i} \mid i \in \cellvar^{c} \>)$ induces a monoid map $\N\genset \to \N\genset_{m}$. By letting $T_{m}$ be the quotient of the saturation of the lattice $L_{m}$ corresponding to $(I:m) \cap \field[\N^{\cellvar}]$ by $L_{m}$, we have a fine grading of $\field[x]/I$ by the abelian group
\[
\bigoplus_{m \in M}(T_{m}\oplus \Z\genset_{m})
\]
via $\deg_{M,T,\genset}(\overline{x^{\vecu}}) = (\vecu^{\cellvar^{c}}, \vecu^{\cellvar}+L_{m}, A_{m}\cdot \vecu^{\cellvar}).$ Thus, for a fixed $m \in M$ and a torsion $t \in T_{m}$, let 
\[
I_{t}:=\<x^{\genset_{m} \cdot \vecu^{\cellvar}} \mid \deg_{M,T,\genset}(\overline{x^{\vecu}}) = (\deg(m),t,A\cdot \vecu^{\cellvar}) \text{ and } \overline{x^{\vecu}} \in \field[x]/I \>.
\]
Using the same arguments as for lattice ideals, we know that two elements whose degrees are in the same minimal open set $O \in \mathcal{G}(\field[\N\genset_{m}]/I_{t})$  have the same graded part of the Ishida complex  $\rchaincpx{O}$. Then, we have
\begin{theorem}
\label{thm:hochster_cellular}
Given a cellular binomial ideal $I$, the multi-graded Hilbert series for the local cohomology module of $\field[x]/I$ supported on the image of the radical monomial ideal $\ridI \subset \field[\affinesemigp]$ with respect to the $\bigoplus_{m \in M}(T_{m}\oplus \Z\genset_{m})$-grading is
\[
\hilb{\localcoho{\ridI}{i}{\field[x]/I}}{\hilbt}=\sum_{m \in M}\sum_{t \in T_{m}} \sum_{O \in \mathcal{G}(\field[\N\genset_{m}]/I_{t})}\dim H^{i}(\rchaincpx{O};\field) \sum_{\substack{u \in \Z^{d} \\ \genset_{m}\cdot u^{\cellvar} \in O}}\overline{x^{u}}.
\]
\end{theorem}
Again, this is a finite sum of rational functions.
The Corollary~\ref{cor:reisner} gives us the equivalent of Reisner's criterion for cellular binomial ideals, providing a characterization of Cohen-Macaulayness in terms of the cohomology of finitely many polyhedral complexes. First we need an auxilliary result.
\begin{lemma}
$\rchaincpx{O}$ is the cochain complex of a polyhedral complex.
\end{lemma}
\begin{proof}
It suffices to show that the nontrivial top dimensional part of $\rchaincpx{O}$ is $\field^{1}$; suppose not; then there exists distinct maximal faces $\virtual{\facF_{1}}$ and $\virtual{\facF_{2}}$ of $\affinesemigp$ such that $\deg_{\genset}(\overline{x^{u_{1}}}) \in \virtual{\facF_{1}}$ and $\deg_{\genset}(\overline{x^{u_{2}}}) \in \virtual{\facF_{2}}$ for some distinct $\overline{x^{u_{1}}}$ and $\overline{x^{u_{2}}}$ with $u_1,u_2 \in O$. Then, the degree of the product $\overline{x^{u_{1}}} \cdot \overline{x^{u_{2}}}$  lies in the relative interior of a face $\virtual{\facG}$ in $\facelattice(\affinesemigp)$ which is a minimal face containing both $\virtual{\facF_{1}}$ and $\virtual{\facF_{2}}$, contradicting the maximality of $\virtual{\facF_{1}}$ and $\virtual{\facF_{2}}$.
\end{proof}

\begin{corollary}\label{cor:reisner}
Let $I$ be a cellular binomial ideal. Then $\field[x]/I$ is Cohen--Macaulay if and only if $H^i (\rchaincpx{O};\field) = 0$ for all $i \neq \text{dim}(\field[x]/I)$ and for all $O \in \mathcal{G}(\field[\N\genset_{m}]/I_{t})$, $m \in M, t \in T_m$. \qed
\end{corollary}

\begin{example}
\label{ex:lattice_ideal}

Let $L$ be the following lattice in $\Z^{4}$ and $L_{\text{sat}}$ its saturation. 
\[
L:= \left\langle \left(\begin{smallmatrix} 2 \\ 0 \\ -3 \\ 0 \end{smallmatrix}\right), \left(\begin{smallmatrix} 1 \\ -5 \\ 1 \\ 5 \end{smallmatrix}\right) \right\rangle,L_{\text{sat}}:=  \left\langle \left(\begin{smallmatrix} 1 \\ -1 \\ -1 \\ 1 \end{smallmatrix}\right), \left(\begin{smallmatrix} 0 \\ -2 \\ 1 \\ 2 \end{smallmatrix}\right),\left(\begin{smallmatrix} -1 \\ -1 \\ 2 \\ 2 \end{smallmatrix}\right),\left(\begin{smallmatrix} -2 \\ 0 \\ 3 \\ 0 \end{smallmatrix}\right) \right\rangle.
\] 
The torsion group $T:=L/L_{\text{sat}}$ is isomorphic to $\Z/5\Z$. We represent $T$ as follows
\[
T=\left\{ e = \overline{\left(\begin{smallmatrix} 0 \\ 0 \\ 0 \\0 \end{smallmatrix}\right)}, \xi=\overline{\left(\begin{smallmatrix} -1 \\ 1 \\ 1 \\ -1 \end{smallmatrix}\right)}, \xi^{2} = \overline{\left(\begin{smallmatrix} 0 \\ 2 \\ -1 \\ -2 \end{smallmatrix}\right)}, \xi^{3} = \overline{\left(\begin{smallmatrix} -1 \\ 3 \\ 0 \\-3 \end{smallmatrix}\right)}, \xi^{4} = \overline{\left(\begin{smallmatrix} 0 \\ 4 \\ -2 \\-4 \end{smallmatrix}\right)}  \right\}
\]
using GRevLex term order in Macaulay2.

In this case $Q:=(\Z^{4}/L_{\text{sat}}) \cap \N^{4}\cong \N \left[\begin{smallmatrix}3 & 1 & 2 & 0 \\ 0 & 1 & 0 & 1 \end{smallmatrix}\right]$. As usual $\Z Q = \Z^{4}/L_{\text{sat}}$. On the polynomial ring $\field[a,b,c,d]$, the lattice ideals corresponding to $L$ and $L_{\text{sat}}$ are 
$$I_{L}:=\langle a^{2}-c^{3}, acd^{5}-b^{5} \rangle \text{ and } I_{\text{sat}}:=\langle bc-ad,cd^{2}-b^{2},c^{2}d-ab,c^{3}-a^{2} \rangle.$$
Hence, the Ishida complex supported on the maximal monomial ideal is 
$$0 \to \field[x]/I_{L} \to (\field[x]/I_{L})_{a,c}\oplus (\field[x]/I_{L})_{d} \to (\field[x]/I_{L})_{a,b,c,d} \to 0.$$
Here, we see 
that for any $0 \leq i$, 
\[
C_{t,\left[\begin{smallmatrix}i \\ j\end{smallmatrix}\right]}^{\bullet}: 0\to \mathbb{K}^{j+1} \to \mathbb{K}^{j+1 +5} \to \mathbb{K}^{5} \to 0
\]
when $0 \leq j <5$ and
\[
C_{t,\left[\begin{smallmatrix}i \\ j\end{smallmatrix}\right]}^{\bullet}: 0\to \mathbb{K}^{5} \to \mathbb{K}^{10} \to \mathbb{K}^{5} \to 0
\]
when $5 \leq j$ for any $t \in T$. 
Since there is no non-top cohomology, we
conclude that $\field[\varx]/I_{L}$ is Cohen-Macaulay. 
\end{example}

\section{Local cohomology for affine semigroups over simplices}
\label{sec:hochster_thm}

In this section, we consider affine semigroup rings such that $\realize \affinesemigp$ is a cone over a simplex. In this case the combinatorics simplifies and becomes more explicit.

\subsection{Hyperplane arrangements}
\label{subsec:hyperplane}

The \emph{hyperplane arrangement} $\arngmt:=\{\hplane_{1},\cdots, \hplane_{m} \}$ of a polyhedron $\polytope$ is the collection of supporting hyperplanes of the facets of $\polytope$ in $\R^{d}$. In this case, $\bigcap_{i=1}^{m}\hplane^{+} = \polytope$, where 
$\hplane^{+}$ denotes the positive half-space associated to a hyperplane (the positive side is the side that contains our polyhedron). A hyperplane arrangement is \emph{linear} if all hyperplanes in the arrangement contain the origin. A \emph{region} $\regr$ of a hyperplane arrangement $\arngmt$ is a connected component of $\R^{d}-\bigcup_{\hplane \in \arngmt} \hplane$. $\regionr{\arngmt}$ refers to the collection of all regions of $\arngmt$. 

Suppose $\arngmt$ consists of a minimal number of hyperplanes which generate a rational polyhedral cone $\polytope$. 
Then $\arngmt$ is linear and all regions in $\regionr{\arngmt}$ are unbounded rational polyhedral cones. Moreover, every region $\regr$ can be expressed as
\[
\regr_{\setS}:=\left(\bigcap_{\smalli \in [m]\minus\setS} \hplane_{\smalli}^{+}\right) \cap \left(\bigcap_{\smalli \in \setS} \hplane_{\smalli}^{-}\right)\minus \bigcup_{\smalli=1}^{m}\hplane_{\smalli}
\]
for a subset $\setS \subseteq [m]$ where $\hplane_{\smalli}^{-}$ is the complement of $\hplane_{\smalli}^{+}$. In other words, a region is labeled by the collection of hyperplanes whose positive half space contains it. 
It follows that $\regionr{\arngmt}$ is partially ordered by reverse inclusion on the set of labels; $\regr_{\setS_{1}} \leq \regr_{\setS_{2}} \text{ if } \setS_{1} \supseteq \setS_{2}.$ We call $\regionr{\arngmt}$ the \emph{poset of regions} of $\arngmt$. Our notation here is consistent with that of ~\cites{Edelman84,BEZ90}  regarding $\polytope$ as the base region. Moreover, the natural embedding $\facelattice(\polytope) \to \regionr{\arngmt}$ sending a face $\facF$ to the set of indices of hyperplanes containing $\facF$ exists~\cite{Edelman84}*{Lemma 1.3}.  Since we are interested in regions partitioning $\R^{d}$ along with the set of degrees of standard monomials of localizations, we modify the definition of $\regr_{\setS}$ as follows, to include boundaries:
\[
\regr_{\setS}:=\left(\bigcap_{\smalli \in [m] \minus \setS} \hplane_{\smalli}^{+}\right) \cap \left(\bigcap_{\smalli \in \setS} \hplane_{\smalli}^{-}\right),
\]

A \emph{cumulative region} $\regR_{\setS} = \left(\bigcap_{\smalli \in [m]\setminus \setS} \hplane_{\smalli}^{+}\right)$ is the union of all regions less then $\regr_{\setS}$. The \emph{poset of cumulative regions} $\regionR{\arngmt}$ is a set of all cumulative regions ordered by inclusion. By definition, $\regionR{\arngmt} \cong \regionr{\arngmt}$ as posets. We follow the conventions of~\cite{Stanley07}.

\subsection{Sections of polyhedra}
\label{subsec:polyhedra}

If $\facF$ is a zero-dimensional face of a polyhedral complex $\polytope$, then we define the \emph{vertex figure of $\polytope$ at $\facF$} as follows. Assume $\polytope$ is realized in $\R^{\dimd}$, 
Pick a sphere $S^{\dimd-1}$ centered at $\facF$ such that every nonempty face of $\polytope$ except $\facF$ is not completely contained in the sphere. 
Then, the vertex figure $\polytope/\facF$ is the polyhedral complex generated by the intersections of all faces of $\polytope$ containing $\facF$ on $S^{\dimd-1}$. Likewise, when $\facF$ is positive dimensional face, we define the \emph{section of $\polytope$ at $\facF$}, denoted $\polytope/\facF$, to be the polytope generated by taking vertex figures iteratively over the vertices arising from $\facF$. As a new abstract polytope, $\facG/\facF$ is combinatorially equivalent to the link of $\facF$ over $\facG$ as a sub-polyhedral complex of $\polytope$.
Moreover, if we regard the face lattice of the vertex figure as a collection of faces of $\polytope$, then its closure in the usual Euclidean topology agrees with the star of $\facF$.

Given a polyhedral complex $\Delta$, let $\max(\Delta)$ be the set of maximal elements of $\Delta$, and let $\bigcap\max(\Delta)$ be a set of faces which are intersections of maximal faces of $\Delta$ as a poset.
We say $\Delta$ is $m$-\emph{combinatorially connected} for $m:=\min_{\facF \in \bigcap\max(\Delta)}\dim \facF$. 
Note that this is a much finer notion of connectedness than the usual topological $n$-connectedness. For example, a simplicial complex consisting of two triangles sharing an edge is $1$-combinatorially connected but contractible (infinitely-connected) in the sense of topological $n$-connectedness. 

\begin{lemma}
\label{lem:m_connectivity_and_contractibility}
Any vertex figure of an $m$-combinatorially connected polyhedral complex $\Delta$ is at least $(m-1)$-combinatorially connected. Moreover, when $m \geq1$, for any vertex $\facF \in \Delta$, the vertex figure $(\Delta/\facF)$ is contractible. Hence, for any face $\facF \in \Delta$ whose dimension is less than $m$, $(\crosssection/\facF) \cap \Delta$ is contractible.
\end{lemma}

\begin{proof}
As the $m=0$ case holds vacuously, we may let $m \geq 1$. Pick a vertex $\facF$. Let $\facG$ be the intersection of all maximal faces of $\Delta$ containing $\facF$. Then, $\dim \facG \geq m.$ All maximal faces of the vertex figure $\Delta /\facF$ are inherited from those maximal faces of $\Delta$ containing $\facF$, hence $\Delta /\facF$ is $(\dim\facG-1)$-combinatorially connected, and $\dim\facG-1 \geq m-1$. 

For the second statement, let $\facF$ and $\facG$ be the same as above. Let $X= \{ \facH \in \Delta \mid  \facH \supseteq \facF \}$ as a set of faces of $\Delta$. 
Since $X$ is homotopic to $\facG$; we may collapse each maximal face in $X$ containing $\facG$ continuously to $\facG$. Moreover, $X \minus \{ \facF\}$ is homotopic to $\facG \minus \{ \facF\}$. To see this, let $S_{\facF}$ be a sphere centered at $\facF$ and generating the vertex figure of $\Delta$ and $\facG$ on its surface. The homotopy from $X$ to $\facG$ restricted on $S_{\facF}$ gives the homotopy between vertex figures of $X$ and $\facG$ over $\facF$ if $\dim \facG \geq 1$. The last statement follows by applying the second statement iteratively.
\end{proof}

\begin{lemma}
\label{lem:classification_of_vertex_figure}
Given a $\dimd$-dimensional polyhedral complex $\polytope$ homeomorphic to a disk $D^{\dimd}$, let $\facF$ be a vertex. 
If $\facF$ is in the interior of $\polytope$, then $\polytope/\facF$ is homeomorphic to $S^{\dimd-1}$. Otherwise, if $\facF$ is in the boundary of $\polytope$, then $\polytope/\facF$ is homeomorphic to $D^{\dimd-1}$.
\end{lemma}

\begin{proof}
If $\facF$ is in the boundary of $\polytope$, the result is clear since the sphere containing $\facF$ that defines the vertex figure cannot be contained in the polyhedral complex. One can apply a similar argument to the case when $\facF$ is in the relative interior.
\end{proof}

\subsection{Local cohomology}
\label{subsec:hochster}

Let $(I,I,\genset)$ be a prime lattice ideal $I$ whose corresponding affine semigroup $\affinesemigp = \N\genset$ is \emph{simplicial}, i.e., the transverse section $\crosssection$ of the polyhedral cone $\R_{\geq 0}\affinesemigp$ is a $d$-simplex. Let $\arngmt=\{\hplane_{i}\}_{i=1}^{\dimd}$ be the minimal hyperplane arrangement of $\R_{\geq 0}\affinesemigp$. Label the faces of $\crosssection$ by their supporting facets; for a facet $\hplane_{i} \cap \crosssection$, use the label $[\dimd]\minus \{ i\}$. Hence, the zero-dimensional face (of the transverse section) that does not lie in the hyperplane $\hplane_{i}$ is indexed by $\{ i\}.$ 

From the natural isomorphism $\facelattice(\affinesemigp) \cong \facelattice(\crosssection)$, we may label faces of $\affinesemigp$ using $2^{[d]}=\facelattice(\crosssection)$.

In this notation, for a face $\facF \in \facelattice(\affinesemigp)$, $\regr_{\facF}$ defined in~\cref{subsec:hyperplane} by regarding $\facF$ as a subset of $[d]$, is a region generated by the positive half spaces containing $\facF$ as a face of $\affinesemigp$. This labeling is induced by the observation that the poset of regions $\regionr{\arngmt}$ is equal to the face lattice $\facelattice(\crosssection)$. This also agrees with the labeling of $\regr_{\facF}$ in~\cref{subsec:hyperplane}; $\regr_{\facF}$ is the region contained in the positive half spaces of $\hplane_{i}$ where $i \in \facF$. For example, if $\facv$ is the zero-dimensional face corresponding to $x_{i}$, then $\regr_{\facv}= \regr_{[d] \minus \{ i\}}$. 

Let $\ridI$ be a radical monomial ideal of $\field[\affinesemigp] \cong \field[x]/I$ corresponding to a proper subcomplex $\Delta$ of $\crosssection$ and a proper face $\facF \in \facelattice(\affinesemigp)$. In this setting, we always label all minimal open sets of the original degree space $\bigcup\deg(\field[\affinesemigp]/\ridI)$ using a pair of faces as in the following result. We denote $\max(\Delta)$ the collection of facets of $\Delta$.

\begin{lemma}
\label{lem:index_minimal_open_sets}
For every minimal open set $\setS \in \mathcal{G}(\field[\affinesemigp]/I)$ of the degree space $\bigcup\deg(\field[\affinesemigp]/\ridI)$ there exists a unique pair of faces $(\facF, \facG)$ such that $\setS \subseteq (\affinesemigp-\N\facF) \cap \regr_{\facF}$ and $\facG \in \{ \facG' \in \bigcap \max(\Delta) \mid \facG' \supseteq \facG \supseteq \facF \}$. We use this pair to label $\setS$.
\end{lemma}

\begin{proof}
Given a face $\facF \in \facelattice(\affinesemigp)$, let $\mathcal{G}(\affinesemigp/\ridI)_{\facF}$ be the set of minimal open sets of the degree space $\bigcup\deg(\field[x]/\ridI)$ contained in $\regr_{\facF}$. Recall that $\bigcap \max(\Delta)_{\facF}$ is the collection of faces in $\bigcap \max(\Delta)$ containing $\facF$. We claim that $\mathcal{G}(\affinesemigp/\ridI)_{\facF}$ and $\bigcap \max(\Delta)_{\facF}$ are in bijection. 

Recall that all overlap classes (for $\ridI$) are of the form $[0,\facG]$ for some face $\facG \in \max(\Delta)$. Hence, when $\facF = \tilde{0},$ the corresponding minimal open set inside of $\regr_{\emptyset}\cap\affinesemigp= \affinesemigp$ is obtained by  intersection and complement of faces. If $\facF \not\in \Delta$, then the localization is zero, thus the statement is vacuously true. Suppose $\facF \gneq \tilde{0} \in \Delta$. The extension $(\ridI)_{\facF}$ of the ideal on $\field[\affinesemigp-\N\facF]$ is still radical and its overlap classes are of form $[0,\facG \cup (-\facF)]$ where $[0,\facG]$ is an overlap class of $\field[\affinesemigp]/\ridI$ for some face $\facG \supseteq \facF$. Thus, every minimal open set in $\mathcal{G}(\affinesemigp/\ridI)_{\facF}$ is obtained by intersecting open sets of the form $\bigcup[0,\facG \cup(-\facF)] \minus \left(\bigcup_{\facF' \subset \facG}\regr_{\facF'} \right)$. 
Thus $\mathcal{G}(\affinesemigp/\ridI)_{\facF}$ is labeled by $\bigcap \max(\Delta)_{\facF}$.
\end{proof}

We are now ready to study graded pieces of local cohomology modules.

\begin{theorem}
\label{thm:non_acyclic_chains_of_local_coho_of_maximal_id} 
Let $\vecu$ be a degree in a minimal open set $\setS$ indexed by $(\facF,\facG)$ with $\facF \neq \facG$. Then, $\localcoho{\maxid}{\smalli}{\field[\affinesemigp]/\ridI}_{\vecu}=0$ for all $i$. If $\vecu$ is a degree in a minimal open set indexed by a pair of the same face $(\facF,\facF)$, 
\[
\localcoho{\maxid}{\smalli}{\field[\affinesemigp]/\ridI}_{\vecu} \cong \tilde{H}_{\text{simp}}^{\smalli-\dim \facF-1}((\crosssection/\facF) \cap \Delta)
\]
where $\tilde{H}_{\text{simp}}^{\bullet}(-)$ means the reduced simplicial cohomology.
\end{theorem}

When our affine semigroup ring is the polynomial ring, the above reduces to the very 
well known formulas for local cohomology of Stanley--Reisner rings using homology of links.

\begin{proof}[Proof of~\cref{thm:non_acyclic_chains_of_local_coho_of_maximal_id}]
Pick a minimal open set $\setS$ that corresponds to $(\facF,\facG)$. 
We know that $\deg(\setS) \subset \field[\affinesemigp-\facH]$, where $\facH \in  (\facF/\facG)$.
We may assume $\facG = \bigcap_{i=1}^{m}\facG_{i}$ for some $\facG_{i} \in \max(\Delta)$ containing $\facG$. Then $\setS$ is contained in an overlap class of $\ridI$ whose face is $\facG_{i}$. From~\cref{def:gen_ishida}, the $\vecu$-graded part of the Ishida complex is equal to the (shifted) chain complex of $\Delta(\facG_{1},\cdots, \facG_{m})/\facF$ where $\Delta(\facG_{1},\cdots, \facG_{m})$ is a subcomplex of $\Delta$ such that its maximal faces are $\facG_{1},\cdots,\facG_{m}$. When $\facF \neq \facG$,~\cref{lem:m_connectivity_and_contractibility} shows that $\Delta(\facG_{1},\cdots, \facG_{m})/\facF$ is contractible. Otherwise, $\Delta(\facG_{1},\cdots, \facG_{m})/\facF= \Delta\cap (\crosssection/\facF)$.
\end{proof}

We remark that this theorem holds more generally, for any affine semigroup $\affinesemigp$ 
whose poset of regions $\regionr{\arngmt}$ is in bijection with $\facelattice(\affinesemigp)$. 

\begin{corollary}
\label{cor:non_acyclic_chains_of_local_coho_of_maximal_id} 
\cref{lem:index_minimal_open_sets} and~\cref{thm:non_acyclic_chains_of_local_coho_of_maximal_id} hold when $\affinesemigp$ is an affine semigroup such that $\facelattice(\affinesemigp)$ is in bijection to the poset of regions $\regionr{\arngmt}$ of the hyperplane arrangement of $\realize\affinesemigp$.
\end{corollary}
\begin{proof}
The property we used in the proof of~\cref{lem:index_minimal_open_sets} is that for any $\vecu$, there exists a unique minimal face $\facF$ such that $\vecu \in \affinesemigp-\facF$. This property holds if and only if $\regionr{\arngmt}$ is in bijection to $\facelattice(\affinesemigp)$. 
\end{proof}

\begin{figure*}[t!]
\centering
\tdplotsetmaincoords{90}{90} 
\begin{tikzpicture}[tdplot_main_coords, scale=1]
\tdplotsetrotatedcoords{40}{-10}{30}
\foreach \x in {0,...,4}
	{
		\draw[step=1cm,gray,very thin, tdplot_rotated_coords] (\x,0,0) -- (\x,4,0);
		\draw[step=1cm,gray,very thin, tdplot_rotated_coords] (0,\x,0) -- (4,\x,0);
		\draw[step=1cm,gray,very thin, tdplot_rotated_coords] (\x,0,1) -- (\x,4,1);
		\draw[step=1cm,gray,very thin, tdplot_rotated_coords] (0,\x,1) -- (4,\x,1);
		\draw[step=1cm,gray,very thin, tdplot_rotated_coords] (\x,0,2) -- (\x,4,2);
		\draw[step=1cm,gray,very thin, tdplot_rotated_coords] (0,\x,2) -- (4,\x,2);
		\draw[step=1cm,gray,very thin, tdplot_rotated_coords] (\x,0,3) -- (\x,4,3);
		\draw[step=1cm,gray,very thin, tdplot_rotated_coords] (0,\x,3) -- (4,\x,3);
		\draw[step=1cm,gray,very thin, tdplot_rotated_coords] (\x,0,4) -- (\x,4,4);
		\draw[step=1cm,gray,very thin, tdplot_rotated_coords] (0,\x,4) -- (4,\x,4);

		\draw[step=1cm,gray,very thin, tdplot_rotated_coords] (0,\x,0) -- (0,\x,4);
		\draw[step=1cm,gray,very thin, tdplot_rotated_coords] (0,0,\x) -- (0,4,\x);
		\draw[step=1cm,gray,very thin, tdplot_rotated_coords] (1,\x,0) -- (1,\x,4);
		\draw[step=1cm,gray,very thin, tdplot_rotated_coords] (1,0,\x) -- (1,4,\x);
		\draw[step=1cm,gray,very thin, tdplot_rotated_coords] (2,\x,0) -- (2,\x,4);
		\draw[step=1cm,gray,very thin, tdplot_rotated_coords] (2,0,\x) -- (2,4,\x);
		\draw[step=1cm,gray,very thin, tdplot_rotated_coords] (3,\x,0) -- (3,\x,4);
		\draw[step=1cm,gray,very thin, tdplot_rotated_coords] (3,0,\x) -- (3,4,\x);
		\draw[step=1cm,gray,very thin, tdplot_rotated_coords] (4,\x,0) -- (4,\x,4);
		\draw[step=1cm,gray,very thin, tdplot_rotated_coords] (4,0,\x) -- (4,4,\x);
	};
	
	\draw[thick,->,tdplot_rotated_coords] (0,0,0) -- (6,0,0) node[anchor=south west]{$x$}; 
	\draw[thick,->,tdplot_rotated_coords] (0,0,0) -- (0,5.5,0) node[anchor=south east]{$y$}; 
	\draw[thick,->,tdplot_rotated_coords] (0,0,0) -- (0,0,5.5) node[anchor=south]{$z$};
	\draw[thick,->,tdplot_rotated_coords] (0,0,0) -- (0,4.5,4.5);
	\draw[thick,dashed,->,tdplot_rotated_coords] (0,0,0) -- (4.5,0,4.5);
	\draw[thick, ->,tdplot_rotated_coords] (0,0,0) -- (4.5,4.5,4.5);

	\foreach \y in {0,...,4}
	\draw[black,fill=stdmcolor,tdplot_rotated_coords] (0,0,\y) circle (2 pt);
	\draw[black,fill=stdmcolor,tdplot_rotated_coords] (0,1,1) circle (2 pt);
	\draw[black,fill=stdmcolor,tdplot_rotated_coords] (0,1,2) circle (2 pt);
	\draw[black,fill=stdmcolor,tdplot_rotated_coords] (0,2,2) circle (2 pt);
	\draw[black,fill=stdmcolor,tdplot_rotated_coords] (0,1,3) circle (2 pt);
	\draw[black,fill=stdmcolor,tdplot_rotated_coords] (0,2,3) circle (2 pt);
	\draw[black,fill=stdmcolor,tdplot_rotated_coords] (0,3,3) circle (2 pt);
	\draw[black,fill=stdmcolor,tdplot_rotated_coords] (0,1,4) circle (2 pt);
	\draw[black,fill=stdmcolor,tdplot_rotated_coords] (0,2,4) circle (2 pt);
	\draw[black,fill=stdmcolor,tdplot_rotated_coords] (0,3,4) circle (2 pt);
	\draw[black,fill=stdmcolor,tdplot_rotated_coords] (0,4,4) circle (2 pt);

	\foreach \y in {0,...,3}
	\draw[black,fill=stdmcolor,tdplot_rotated_coords] (1,0,1+\y) circle (2 pt);

	\foreach \y in {0,...,3}
	\draw[black,fill=stdmcolor,tdplot_rotated_coords] (1,1+\y,1+\y) circle (2 pt);
	
	\draw[black,fill=stdmcolor,tdplot_rotated_coords] (2,2,2) circle (2 pt);

	\draw[black,fill=stdmcolor,tdplot_rotated_coords] (1,1,2) circle (2 pt);
	\draw[black,fill=stdmcolor,tdplot_rotated_coords] (1,1,3) circle (2 pt);
	\draw[black,fill=stdmcolor,tdplot_rotated_coords] (1,2,3) circle (2 pt);
	\draw[black,fill=stdmcolor,tdplot_rotated_coords] (1,1,4) circle (2 pt);
	\draw[black,fill=stdmcolor,tdplot_rotated_coords] (1,2,4) circle (2 pt);
	\draw[black,fill=stdmcolor,tdplot_rotated_coords] (1,3,4) circle (2 pt);

	\draw[black,fill=stdmcolor,tdplot_rotated_coords] (2,0,2) circle (2 pt);
	\draw[black,fill=stdmcolor,tdplot_rotated_coords] (2,1,2) circle (2 pt);
	
	\draw[black,fill=stdmcolor,tdplot_rotated_coords] (2,0,3) circle (2 pt);
	\draw[black,fill=stdmcolor,tdplot_rotated_coords] (2,1,3) circle (2 pt);
	\draw[black,fill=stdmcolor,tdplot_rotated_coords] (2,2,3) circle (2 pt);
	\draw[black,fill=stdmcolor,tdplot_rotated_coords] (2,3,3) circle (2 pt);
	\draw[black,fill=stdmcolor,tdplot_rotated_coords] (3,0,3) circle (2 pt);
	\draw[black,fill=stdmcolor,tdplot_rotated_coords] (3,1,3) circle (2 pt);
	\draw[black,fill=stdmcolor,tdplot_rotated_coords] (3,2,3) circle (2 pt);
	\draw[black,fill=stdmcolor,tdplot_rotated_coords] (3,3,3) circle (2 pt);
	\draw[black,fill=stdmcolor,tdplot_rotated_coords] (2,0,4) circle (2 pt);
	\draw[black,fill=stdmcolor,tdplot_rotated_coords] (2,1,4) circle (2 pt);
	\draw[black,fill=stdmcolor,tdplot_rotated_coords] (2,2,4) circle (2 pt);
	\draw[black,fill=stdmcolor,tdplot_rotated_coords] (2,3,4) circle (2 pt);
	\draw[black,fill=stdmcolor,tdplot_rotated_coords] (2,4,4) circle (2 pt);
	\draw[black,fill=stdmcolor,tdplot_rotated_coords] (3,0,4) circle (2 pt);
	\draw[black,fill=stdmcolor,tdplot_rotated_coords] (3,1,4) circle (2 pt);
	\draw[black,fill=stdmcolor,tdplot_rotated_coords] (3,2,4) circle (2 pt);
	\draw[black,fill=stdmcolor,tdplot_rotated_coords] (3,3,4) circle (2 pt);
	\draw[black,fill=stdmcolor,tdplot_rotated_coords] (3,4,4) circle (2 pt);
	\draw[black,fill=stdmcolor,tdplot_rotated_coords] (4,0,4) circle (2 pt);
	\draw[black,fill=stdmcolor,tdplot_rotated_coords] (4,1,4) circle (2 pt);
	\draw[black,fill=stdmcolor,tdplot_rotated_coords] (4,2,4) circle (2 pt);
	\draw[black,fill=stdmcolor,tdplot_rotated_coords] (4,3,4) circle (2 pt);
	\draw[black,fill=stdmcolor,tdplot_rotated_coords] (4,4,4) circle (2 pt);

\end{tikzpicture}
\caption{Degrees of $Q = \N \left[\begin{smallmatrix}0&1&0&1\\0&0&1&1\\ 1&1&1&1 \end{smallmatrix}\right]$}
\label{fig:3d_stdpairs}
\end{figure*}

\begin{figure*}[t!]
\centering
\begin{tikzpicture}[scale=0.8]
  \node (max) at (0,4) {$\Z\affinesemigp$};
  \node (abd) at (-3,2) {$\affinesemigp-\facF_{1}$};
  \node (abc) at (-1,2) {$\affinesemigp-\facF_{2}$};
  \node (bcd) at (1,2) {$\affinesemigp-\facF_{3}$};
  \node (acd) at (3,2) {$\affinesemigp-\facF_{4}$};
  \node (ad) at (-3,0) {$\affinesemigp-\langle \veca_{1}\rangle$};
  \node (ab) at (-1,0) {$\affinesemigp-\langle \veca_{2}\rangle$};
  \node (bc) at (1,0) {$\affinesemigp-\langle \veca_{3}\rangle$};
  \node (cd) at (3,0) {$\affinesemigp-\langle \veca_{4}\rangle$};
  \node (min) at (0,-4) {$\affinesemigp$};
  \draw (max) -- (abc) (max) -- (abd) (max) -- (acd) (max) -- (bcd);
  \draw (abc)--(ab) (abc) -- (bc) (bcd) -- (bc) (bcd) -- (cd) (abd) -- (ab) (abd) -- (ad);
\draw[preaction={draw=white, -,line width=6pt}] (acd) -- (ad) (acd) --(cd);
  \draw (ad) -- (min) (ab) -- (min)  (bc)--(min) (cd) -- (min);

\draw[->] (4,0) -- (6,0);

  \node (zmax) at (10,4) {$\Z\affinesemigp$};
  \node (zabd) at (7,2) {$\regr_{1}$};
  \node (zabc) at (9,2) {$\regr_{2}$};
  \node (zbcd) at (11,2) {$\regr_{3}$};
  \node (zacd) at (13,2) {$\regr_{4}$};
  \node (zad) at (7,0) {$\regr_{1,4}$};
  \node (zab) at (9,0) {$\regr_{1,2}$};
  \node (zbc) at (11,0) {$\regr_{2,3}$};
  \node (zcd) at (13,0) {$\regr_{3,4}$};
  \node (zd) at (7,-2) {$\regr_{1,3,4}$};
  \node (za) at (9,-2) {$\regr_{1,2,4}$};
  \node (zb) at (11,-2) {$\regr_{1,2,3}$};
  \node (zc) at (13,-2) {$\regr_{2,3,4}$};
  \node (zmin) at (10,-4) {$\affinesemigp$};
  \draw (zmax) -- (zabc) (zmax) -- (zabd) (zmax) -- (zacd) (zmax) -- (zbcd);
  \draw (zabc)--(zab) (zabc) -- (zbc) (zbcd) -- (zbc) (zbcd) -- (zcd) (zabd) -- (zab) (zabd) -- (zad);
\draw[preaction={draw=white, -,line width=6pt}] (zacd) -- (zad) (zacd) --(zcd);
  \draw (zad) -- (za) (zad) -- (zd) (zab) -- (za) (zab) -- (zb) (zbc)--(zb) (zbc) -- (zc);
    \draw[preaction={draw=white, -,line width=6pt}] (zcd)--(zd) (zcd) -- (zc);
      \draw (zmin) -- (za) (zmin) -- (zb) (zmin) -- (zc) (zmin) -- (zd);
\end{tikzpicture}
\caption{Hasse diagrams of $\catname{Cat}_{\affinesemigp}$ and $\regionr{\arngmt}$ of the Segre embedding}
\label{fig:3d_hasse_diagram}
\end{figure*}

\begin{example}[Counterexample: Segre Embedding]
\label{ex:segre}
We consider the affine semigroup, $\affinesemigp =\N \left[\begin{smallmatrix}0 & 1 & 1 & 0 \\ 0 & 0 & 1 & 1 \\ 1 & 1 & 1 & 1 \end{smallmatrix}\right],$ which is depicted in~\cref{fig:3d_stdpairs}. Let $\veca_{i}$ be the $i$-th column of $\left[\begin{smallmatrix}0 & 1 & 1 & 0 \\ 0 & 0 & 1 & 1 \\ 1 & 1 & 1 & 1 \end{smallmatrix}\right].$ Also denote the facets $\facF_{\smalli}$ and the hyperplane arrangement $\arngmt=\{ \hplane_{1},\hplane_{2},\hplane_{3},\hplane_{4}\}$ as below.
\begin{align*}
\facF_{1}&:= \langle \veca_{1},\veca_{2}\rangle, & \facF_{2}&:= \langle \veca_{2},\veca_{3}\rangle, & \facF_{3}&:= \langle \veca_{3},\veca_{4}\rangle, &\facF_{4}&:= \langle \veca_{4},\veca_{1}\rangle \\
\hplane_{1}^{(+)}&:= \{ y > 0\}, & \hplane_{2}^{(+)}&:= \{ z > x\}, & \hplane_{3}^{(+)}&:= \{ z > y\},& \hplane_{4}^{(+)}&:= \{ x > 0\}.
\end{align*}
For any face $\facF$, label $\facF$ by the subset of $\{1,2,3,4\}$ whose corresponding facet contains $\facF$. For example, $\<\veca_{1}\>$ is indexed by $\{1,4\}$. Then we have the desired injection from $\facelattice(\affinesemigp)$ to $\regionr{\arngmt}$ by sending a face $\facF$ to $\regr_{\facF}:=\realize(\affinesemigp-\N\facF) \setminus \bigcup_{\facG <\facF}\realize(\affinesemigp-\N\facG)$. This relationship is depicted in~\cref{fig:3d_hasse_diagram}. Note that this is \textbf{not} a bijection; for example, 
\[
(0,1,0)^{t} = (1,1,1)^{t}-(1,0,1)^{t} = (0,1,1)^{t}-(0,0,1)^{t}
\] is in both $\affinesemigp-\<\veca_{1} \>$ and $\affinesemigp-\<\veca_{2} \>$ but not in $\affinesemigp$. Hence, $(0,1,0)^{t} \in \regr_{1,2,4}$. Therefore, we may not directly apply~\cref{cor:non_acyclic_chains_of_local_coho_of_maximal_id}. However, still we may apply~\cref{cor:reisner} to calculate its local cohomology by investigating the graded parts of the generalized Ishida complex corresponding to those ``hidden" regions, i.e., regions in the cokernel of the map $\facelattice(\affinesemigp) \to \regionr{\arngmt}$.
\end{example}

\section{Local cohomology duality for simplicial affine semigroup rings}
\label{sec:duality}

In this section, we relate the local cohomology of an affine semigroup ring $\field[\affinesemigp]$ supported on a radical monomial ideal $\ridI$ to the local cohomology supported on the maximal ideal of the quotient of $\field[\affinesemigp]/\ridI$. In the case of Stanley--Reisner rings, it is straightforward to determine the vanishing of such cohomologies using available formulas for (Hilbert series of) local cohomology due to Hochster and to Terai~\cite{Terai99}.

Throughout this section, we use the same notation as in Section~\ref{sec:hochster_thm}.

\subsection{Separated polytope and Alexander duality}
\label{subsec:alexander_duality}

First of all, we rigorously introduce the notion of cutting a face of a polytope. 
Suppose $\polytope$ is a polytope embedded in $\R^{d}$.
Let $\facF$ be a face of $\polytope$ and let $\vecc \in \R^d$ be an outer normal vector for a supporting hyperplane $\hplane_{\facF}$ of $\facF$. We denote this situation $\facF = \face_{\vecc}(\polytope)$.
Let $\hplane_{\facF}'$ be a translation of $\hplane_{\facF}'$ in the direction of $-\vecc$ which separates vertices of $\facF$ and all other vertices of $\polytope$. 
The \emph{separated polytope} $\polytope \minus \facF$ is the polytope defined as the intersection of $\polytope$ and the outer half space of $\hplane_{\facF}'$. 
For example, if $\facF$ is a vertex, then the separated polytope $\polytope\backslash\facF$ is combinatorially equivalent to the vertex figure $\polytope/\facF$. 
We remark that $\facelattice(\polyhedron\minus \facF)$ does not depend on the choice of $\hplane_{\facF}'$. We also refer to the construction of $\polyhedron \minus \facF$ as \emph{cutting the face $\facF$ from $\polyhedron$}.

We recall the statements for both combinatorial and topological Alexander duality below for reference. Note that all simplicial, CW, and singular (co)homology in this article is reduced and with coefficients over $\Z$.

\begin{theorem}[Alexander duality]
\label{thm:Alexander_dual}
For any compact locally contractible nonempty proper topological subspace $K$ of a $(\dimd+1)$-dimensional sphere $S^{\dimd}$, $\tilde{H}_{i}(S^{\dimd} \minus K) \cong \tilde{H}^{\dimd-i-1}(K)$. For any polyhedral subcomplex $\Delta$ of the boundary of a polytope $P$, $\tilde{H}_{i}(\Delta) \cong \tilde{H}^{\dimd-i-3}(\Delta^{\ast})$ where $\Delta^{\ast}:=(\facelattice(\polyhedron) \minus \Delta)^{\text{op}}$ is the \emph{Alexander dual} of $\Delta$, which is a subcomplex of the dual polytope $\polyhedron^{\text{op}}$ of $\polyhedron$.
\end{theorem} 
\begin{proof}
For the topological Alexander duality, refer to~\cite{MR1867354}*{Theorem 3.44}. 
For the combinatorial Alexander duality, refer to~\cites{MR2556456,MR1119198}.
\end{proof}

For a given polytope $\polyhedron$ and a proper polyhedral subcomplex $\Delta$, the \emph{abstract dual polyhedral complex} is the set $\Delta^{\ast,\text{op}}:=(\facelattice(\polyhedron) \minus \Delta)^{\text{op}}$ with the partial order reversed. We use this name because $\Delta^{\ast}$ is an abstract polyhedral complex~\cite{Ziegler95}*{Corollary 2.14}. Let $\rchaincpx{\polyhedron}$ (resp. $\rchaincpx{\Delta}$) be the reduced cochain complex of $\polyhedron$ for the dual polytope. Let $\rchaincpx{\Delta^{\ast,\text{op}}}_{\text{unmoved}}$ be a chain complex which deletes components (and componentwise maps) corresponding to faces not in $\Delta^{\ast,\text{op}}$ from $\rchaincpx{\polyhedron}$. We call $\rchaincpx{\Delta^{\ast,\text{op}}}_{\text{unmoved}}$ the \emph{unmoved Alexander dual chain complex} of $\Delta^{\ast,\text{op}}$. This is still a well-defined chain complex, since it coincides with a reduced chain complex of the Alexander dual of $\Delta$ (with reversed indices) for homology or cohomology. Moreover, for any $i \in \Z$,

\begin{corollary}
\label{cor:unmoved_Alexander_dual}
 $H_{i}(\rchaincpx{\Delta}) \cong H_{i+1}(\rchaincpx{\Delta^{\ast,\text{op}}}_{\text{unmoved}})$ and $H^{i}(\rchaincpx{\Delta}) \cong H^{i+1}(\rchaincpx{\Delta^{\ast,\text{op}}}_{\text{unmoved}})$.
\end{corollary}

\begin{proof}
From~\cref{thm:Alexander_dual}, observe that $H^{\dimd-i-3}(\rchaincpx{\Delta^{\ast}}) \cong H_{i+1}(\rchaincpx{\Delta^{\ast,\text{op}}}_{\text{unmoved}})$ by comparing degrees of their chain complexes. 
The cohomology case is similar.
\end{proof}

\subsection{Duality of graded local cohomologies}
\label{subsec:duality}

Let $\field[\affinesemigp]$ be an affine semigroup ring of dimension $d:=\dim\field[\affinesemigp]$ as defined in~\cref{sec:hochster_thm}. Suppose that $\affinesemigp$ has no hidden regions, i.e., the hyperplane arrangment $\arngmt$ consisting of minimal supporting hyperplanes of $\realize\affinesemigp$ has poset of regions $\regionr{\arngmt}$ canonically bijective to
$\facelattice(\affinesemigp)$~\cite{Edelman84}*{Lemma 1.3}. For example, this is the case when $\realize{\affinesemigp}$ is a cone over a simplex.
Let $\ridI$ be a monomial radical ideal. Let $\crosssection$ be the transverse section of $\R\affinesemigp$ with its index sets defined in~\cref{sec:hochster_thm}. Then, there is a duality between the local cohomology of $\field[\affinesemigp]/\ridI$ with the maximal ideal $\maxid$ support and the local cohomology of $\field[\affinesemigp]$ with $\ridI$-support as follow. 

\begin{theorem}
\label{thm:duality_of_local_cohomologies}
Given a face $\facF \in \facelattice(\affinesemigp)$, let $\setS_{\facF}$ be the minimal open set in the original degree space indexed by $(\facF,\facF)$ from~\cref{lem:index_minimal_open_sets} if it exists. Otherwise, let $\setS_{\facF}$ be $\regr_{\facF}\cap \left( \bigcup\deg(\field[\N\genset]/J_{\Delta})\right)$. 
For any $\vecu \in \setS_{\facF}$ and $\vecv \in \regr_{\facF^{c}} \cap \Z\affinesemigp$,
\[
(\localcoho{\maxid}{\smalli}{\field[\affinesemigp]/\ridI})_{\vecu} \cong (\localcoho{\ridI}{\dimd-\smalli}{\field[\affinesemigp]})_{\vecv}
\]
\end{theorem}

Some duality still holds in the case when hidden regions exist (only degrees outside of hidden regions are involved), but the statement is complicated and not very enlightening. The result is false for degrees corresponding to hidden regions, as seen in the following example.

\begin{example}[Continuation of~\cref{ex:segre}]
\label{ex:counterexample_segre_two}
Let $\ridI=\<x^{\left[\begin{smallmatrix} 0 \\ 1 \\ 2 \end{smallmatrix}\right]},x^{\left[\begin{smallmatrix} 1 \\ 1 \\ 2 \end{smallmatrix}\right]},x^{\left[\begin{smallmatrix} 2 \\ 1 \\ 2 \end{smallmatrix}\right]} \>$. It is a radical monomial ideal such that whose corresponding subcomplex $\Delta$ has $\facF_{1}$ and $\facF_{3}$ as facets. Then, for the grade $(0,1,0)^{t}$, notes that $x^{(0,1,0)^{t}} \in \ridI \cdot \field[\affinesemigp - \N\<\veca_{1}\>] \cap \ridI \cdot \field[\affinesemigp - \N\<\veca_{2}\>]$ and $x^{(0,1,0)^{t}} \not\in \field[\affinesemigp - \N\<\veca_{3}\>] \cup \field[\affinesemigp - \N\<\veca_{4}\>] \cup \field[\affinesemigp - \N\facF_{3}]$. This shows that the graded piece of $L^{\bullet}_{\maxid} \tensor{\field[\affinesemigp]} \field[\affinesemigp]/\ridI$, the Ishida complex of $\field[\affinesemigp]/\ridI$ supported at the monomial maximal ideal $\maxid$, is
\begin{align*}
L^{\bullet}_{\maxid} \tensor{\field[\affinesemigp]} \field[\affinesemigp]/\ridI : 0 \to 0 \to 0 \to 0 \to 0
\end{align*}

On the other hand, the transverse section of $\field[\affinesemigp]\ridI$ is also a rectangle whose vertices are embedded into $\facF_{2}$ or $\facF_{4}$ respectively. Because $(0,1,0)^{t} \in \regr_{1,2,4}$, $x^{(0,1,0)^{t}} \in \field[\affinesemigp-\N\facF_{i}]$ when $i=1,2,4$ only. The graded piece of $L^{\bullet}_{\ridI} \tensor{\field[\affinesemigp]} \field[\affinesemigp]$, the Ishida complex of $\field[\affinesemigp]$ supported at $\ridI$, is therefore
\begin{align*}
L^{\bullet}_{\ridI} \tensor{\field[\affinesemigp]} \field[\affinesemigp] : 0 \to \field^{4} \to \field^{3} \to \field \to 0 
\end{align*}
Consequently,
\[
\left(\localcoho{\ridI}{j}{\field[\affinesemigp]}\right)_{\deg = (0,1,0)^{t}} =\begin{cases} \field & j=1,2 \\ 0 & o.w. \end{cases}, \quad \left(\localcoho{\maxid}{j}{\field[\affinesemigp]/\ridI}\right)_{\deg = (0,1,0)^{t}} =0 \text{ for all }j.
\]
This contradicts the duality when the affine semigroup contains hidden regions.
\end{example}

As a corollary of Theorem~\ref{thm:duality_of_local_cohomologies},
\begin{corollary}
\label{cor:duality_of_local_cohomologies}
If $\affinesemigp$ is a simplicial affine semigroup, then~\cref{thm:duality_of_local_cohomologies} holds.
\end{corollary}

Note that $\facF^{c}:= \{ i \in [n] \mid i\not\in \facF\}$ is a well-defined face since it either exists on both $\facelattice(\affinesemigp)$ and $\regionr{\arngmt}$ simultaneously or on neither.

To proceed, assume that $\ridI$ is neither $0$ nor $\field[\affinesemigp]$. $\mathcal{G}(\field[\affinesemigp])$ consists of $\regr_{\facF} \cap \Z\affinesemigp$ for each face $\facF \in \facelattice(\affinesemigp)$. Given the transverse section $\crosssection_{\ridI}$ of $\realize\ridI$, let $(\crosssection_{\ridI})_{\regr_{\facF^{c}} \cap \Z\affinesemigp}=\{ \facF \in \crosssection_{\ridI} \mid \virtual{\facF} \supset \facF^{c} \}$ be the set of faces of $\crosssection_{\ridI}$ whose corresponding faces in $\facelattice(\affinesemigp)$ contain $\facF^{c}$. Then, $\polyhedron_{\facF^{c}}:= \facelattice( \crosssection_{\ridI}) \minus (\crosssection_{\ridI})_{\regr_{\facF^{c}} \cap \Z\affinesemigp}$ form a polyhedral complex since $\facG \supseteq \facG' \in (\crosssection_{\ridI})_{\regr_{\facF^{c}} \cap \Z\affinesemigp}$ implies $\facG \in (\crosssection_{\ridI})_{\regr_{\facF^{c}} \cap \Z\affinesemigp}$. Consequently, the graded part of the local cohomology of $\field[\affinesemigp]$ with $\ridI$-support is determined as below. 

\begin{lemma}
\label{lem:non_acyclic_chains_of_local_coho_over_radical}
For a degree $\vecu \in \regr_{\facF^{c}} \cap \Z\affinesemigp$,
\[
\localcoho{\ridI}{i}{\field[\affinesemigp]}_{\vecu} \cong (\localcoho{\text{Ishida}}{\smalli-1}{\rchaincpx{\polyhedron_{\facF^{c}}}_{\text{unmoved}}})\cong (\tilde{H}_{\text{CW}}^{\smalli-2}(\polyhedron_{\facF^{c}})).
\]
\end{lemma}

\begin{proof}
The $\vecu$-graded part of the Ishida complex with $\ridI$-support consists of components whose localizations are by faces in $(\crosssection_{\ridI})_{\regr_{\facF^{c}} \cap \Z\affinesemigp}$. Hence, the $\vecu$-graded part of the Ishida complex is equal to the unmoved Alexander dual chain complex of $\polyhedron_{\facF^{c}}$. Therefore, the first isomorphism is from~\cref{cor:unmoved_Alexander_dual}. The second isomorphism is from the difference between homological degrees of $\polyhedron_{\facF^{c}}$ at the Ishida complex and those of $\polyhedron_{\facF^{c}}$ at the CW chain complex.
\end{proof}

By construction, $\polyhedron_{\facF^{c}}$ is a polyhedral complex consisting of faces of $\crosssection_{\ridI}$ whose corresponding faces of $\crosssection$ induce localizations not containing the region $\regr_{\facF^{c}}\cap \Z\affinesemigp$. From a topological viewpoint, $\polyhedron_{\facF^{c}}$ is obtained in two steps; first, cutting out all faces of $\max(\Delta)$ from $\crosssection$ to yield $\crosssection_{\ridI}$. Next, cut out all faces whose corresponding faces of $\crosssection$ inducing localizations containing $\regr_{\facF^{c}}$ from $\crosssection_{\ridI}$ to get $\polyhedron_{\facF^{c}}$. Recall that \emph{cut} is defined rigorously in~\cref{subsec:polyhedra}. We claim that interchanging these two procedures results in a topological space homotopic to $|\polyhedron_{\facF^{c}}|$. 

Let $\crosssection^{\facF^{c}}$ be a simplicial complex whose dual is $(2^{[\dimd]}/\facF^{c})$ on the simplex $\crosssection$. In other words, $\crosssection^{\facF^{c}} := 2^{[\dimd]} \minus (2^{[\dimd]}/\facF^{c})$. $\crosssection^{\facF^{c}}$ is the result of cutting out all faces from $\crosssection$ whose corresponding faces of $\crosssection$ induces localization containing $\regr_{\facF^{c}}$. Now, we cut out all the maximal faces of $\Delta$ from $|\crosssection^{\facF^{c}}|$ if they exist. We claim that the union of those faces cut by this process is the \emph{closure} $\closure{((\crosssection/\facF) \cap \Delta)}$ defined by $\closure{((\crosssection/\facF) \cap \Delta)}:= |\{ \sigma \in 2^{[n]} \mid \sigma \in \tau \text{ for some } \tau  \in (\crosssection/\facF) \cap \Delta\}|$. This is because faces contained in the relative interior of $\crosssection^{\facF^{c}}$ as a topological space are faces containing $\facF$. Hence, cutting a maximal face of $\max(\Delta)$ not containing $\facF$ does not change the combinatorial connectedness of the $\crosssection^{\facF^{c}}$. This argument proves the lemma below.
\begin{lemma}
\label{lem:homotopic_between_original_and_cut_complexes}
If $\facF \neq \crosssection$, $|\polyhedron_{\facF^{c}}|$ is homotopic to the $|\crosssection^{\facF^{c}}| \minus \closure{((\crosssection/\facF) \cap \Delta)}$.
\end{lemma}

Furthermore,

\begin{lemma}
\label{lem:acylic_complexes_of_local_coho_over_radical_ideal}
If $\facF \not\in \bigcap\max(\Delta)$, then $(\localcoho{\ridI}{\bullet}{\field[\affinesemigp]})_{\vecv}=0$ for any $\vecv \in \regr_{\facF^{c}} \cap \Z\affinesemigp$. 
\end{lemma}
\begin{proof}
From~\cref{lem:homotopic_between_original_and_cut_complexes}, it suffices to show that $|\crosssection^{\facF^{c}}| \minus \closure{((\crosssection/\facF) \cap \Delta)}$ is contractible for any $\facF \not\in \bigcap\max(\Delta)$. First, suppose that $\facF$ belongs to $\Delta$ but not to $\bigcap\max(\Delta)$. Then there exists a minimal face $\facG \in \bigcap\max(\Delta)$ containing $\facF$. Now, $\facG$ contains the boundary of $\crosssection^{\facF^{c}}$, thus cutting $\facG$ from $\crosssection^{\facF^{c}}$ does not change its contractibility. Next, suppose that $\facF \not\in \Delta$. Then, no faces of $\bigcap\max(\Delta)$ contain $\facF$, thus if a face of $\bigcap\max(\Delta)$ intersects $\crosssection^{\facF^{c}}$, then the intersection lies on the boundary of $\crosssection^{\facF^{c}}$ as a topological space. This keeps $|\crosssection^{\facF^{c}}|\minus \closure{((\crosssection/\facF) \cap \Delta)}$ contractible.
\end{proof}

Note that in the case when $\facF \neq \tilde{0}$ or $\crosssection$, $\crosssection^{\facF^{c}}$ is homeomorphic to the ball $D^{\dimd-2}$ since it excludes all interior elements of $\crosssection=2^{[\dimd]}$ and ``punctures" the boundary of $\crosssection$. If $\facF = \tilde{0}$, then $\crosssection^{\facF^{c}}$ is homeomorphic to a sphere $S^{\dimd-2}$. If $\facF = \crosssection$, then $\crosssection^{\facF^{c}}$ is an empty set as a $(-1)$-dimensional polyhedral complex.

Now we are ready to show that there is a homotopic image of $\polyhedron_{\facF^{c}}$ which is the dual of $(\crosssection/\facF) \cap \Delta$ in a sphere $S^{(\dimd-1)-\dim\facF-1}$. Recall that $(\crosssection/\facF) \cap \Delta$ as a section can be seen as a subspace of sphere $S^{(\dimd-1)-\dim\facF-1}$ by taking vertex figures iteratively as mentioned in~\cref{subsec:polyhedra}. 

\begin{lemma}
\label{lem:duality_on_vertex_figure}
$\polyhedron_{\facF^{c}}$ is homotopic to $S^{(\dimd-1)-\dim\facF-1} \minus (\crosssection/\facF) \cap \Delta$ where $(\crosssection/\facF) \cap \Delta$ as a $((\dimd-1)-\dim\facF)$-dimensional polyhedral complex realized in $S^{(\dimd-1)-\dim\facF-1}$.
\end{lemma}
\begin{proof}
If $\facF = \tilde{0}$, then $\polyhedron_{\facF^{c}}$ is combinatorially equivalent to the empty set as a polytope and $(\crosssection/\facF)\cap\Delta \cong S^{\dimd-2}$. Also, if $\facF = \crosssection$, then $\Delta = \crosssection$, thus $\polyhedron_{\facF^{c}}= \facelattice(\crosssection) \minus \{ \crosssection\}$ and $(\crosssection/\facF) \cap \Delta = \{ \crosssection\}$. Hence the statement holds for these two cases.

If $\facF$ is nonempty, not maximal nor minimal in $\facelattice(\crosssection)$, recall that $\crosssection^{\facF^{c}}$ is a simplicial complex homotopic to $D^{\dimd-2}$ having $\facF$ in its relative interior. We claim that $\crosssection^{\facF^{c}} \minus \closure{((\crosssection/\facF) \cap \Delta)}$ is homotopic to its image on $S^{\dimd-3}$. To see this, pick a vertex $\facv$ in $\facF$ and take a sphere $S^{\dimd-3}$ centered at $\facv$ but not containing any other vertices. By translation, assume $\facv$ is the origin of $\R^{\dimd-2}$ embedded in $\polyhedron_{\facF^{c}}$. This induces a canonical homotopy map from punctured $\R^{d-2}$ to the sphere $S^{\dimd-3}$ restricted to $\crosssection^{\facF^{c}} \minus \closure{((\crosssection/\facF) \cap \Delta)}$ giving the desired homotopy. Lastly, use~\cref{lem:homotopic_between_original_and_cut_complexes} and~\cref{lem:classification_of_vertex_figure} to conclude that taking vertex figure on $\crosssection^{\facF^{c}}$ preserves its image over a polyhedral complex homeomorphic to $D^{\dimd-2}$ if $\dim\facF \geq 1$, or to $S^{\dimd-3}$ if $\dim\facF = 0$. Iterate this for the other vertices in $\facF$ to complete the argument.
\end{proof}

\begin{corollary}
\label{cor:duality_between_complexes}
For any $i\in \Z$, $\tilde{H}_{\text{CW}}^{i}(\polyhedron_{\facF^{c}}) \cong \tilde{H}_{\text{simp}}^{((\dimd-1)-\dim\facF-1)-i-1}((\crosssection/\facF) \cap \Delta)$.
\end{corollary}
\begin{proof}
\cref{lem:duality_on_vertex_figure} shows $\tilde{H}_{\text{CW}}^{i}(\polyhedron_{\facF^{c}}) \cong \tilde{H}_{\text{CW}}^{i}(S^{(\dimd-1)-\dim\facF-1} \minus (\crosssection/\facF) \cap \Delta)$. Hence, from the topological Alexander duality and the isomorphism between simplicial homology and cohomology, 
\begin{align*}
\tilde{H}_{\text{CW}}^{i}(\polyhedron_{\facF^{c}}) &\cong \tilde{H}_{\text{CW},((\dimd-1)-\dim\facF-1)-i-1}((\crosssection/\facF) \cap \Delta)\cong \tilde{H}_{\text{simp},((\dimd-1)-\dim\facF-1)-i-1}((\crosssection/\facF) \cap \Delta)\\
&\cong \tilde{H}_{\text{simp}}^{((\dimd-1)-\dim\facF-1)-i-1}((\crosssection/\facF) \cap \Delta).
\end{align*}
\end{proof}

\begin{corollary}
\label{cor:duality_of_local_cohomologies_2}
Let $\vecu \in \setS_{\facF}$ where $\setS_{\facF} \in \mathcal{G}(\field[x]/\ridI)$ indexed by $(\facF,\facF)$. Then, for any $\vecv \in \regr_{\facF^{c}} \cap \Z\affinesemigp$, 
\[
\localcoho{\maxid}{\smalli}{\field[\affinesemigp]/\ridI}_{\vecu} \cong \localcoho{\ridI}{\dimd-\smalli}{\affinesemigp}_{\vecv}.
\]
\end{corollary}
\begin{align*}
\localcoho{\maxid}{\smalli}{\field[\affinesemigp]/\ridI}_{\vecu} & \underbrace{\cong}_{\cref{thm:non_acyclic_chains_of_local_coho_of_maximal_id}} \tilde{H}_{\text{simp}}^{\smalli-\dim \facF-1}((\crosssection/\facF) \cap \Delta)\\
& \underbrace{\cong}_{\cref{cor:duality_between_complexes}} \tilde{H}_{\text{CW}}^{ (d-1)-\dim\facF - 1 -(\smalli-\dim \facF-1) -1}(\polyhedron_{\facF^{c}}) \\ 
& \underbrace{\cong}_{\cref{lem:non_acyclic_chains_of_local_coho_over_radical}} \localcoho{\ridI}{ (d-1)-\dim\facF - 1 -(\smalli-\dim \facF-1) -1+2}{\field[Q]}_{\vecv}  \cong \localcoho{\ridI}{ d-i}{\field[Q]}_{\vecv}.
\end{align*}

We are finally ready to prove the main result of this section.

\begin{proof}[Proof of~\cref{thm:duality_of_local_cohomologies}]
If $\Delta=\crosssection$, then $\ridI=0$. Hence $\localcoho{\maxid}{i}{\field[\affinesemigp]/\ridI} =0$ except for $i= \dimd$, also $\localcoho{0}{i}{\field[\affinesemigp]} =0$ except for $i=0$. Moreover, by computation, $\localcoho{\maxid}{\dimd}{\field[\affinesemigp]/\ridI} = \field[\regr_{\emptyset}\cap\bigcup\deg(\field[x])]$ and $\localcoho{0}{0}{\field[\affinesemigp]}=\field[\affinesemigp] =\field[\regr_{[d]} \cap \bigcup\deg(\field[x])]$. Therefore, the desired duality holds. 

If $\Delta = 0$, then $\ridI = \maxid$. Hence, from the generalized Ishida complex $\localcoho{\maxid}{i}{\field[\affinesemigp]/\maxid} =0$ except for $i=0$; in case of $i=0$, $\localcoho{\maxid}{i}{\field[\affinesemigp]/\ridI} =\field = \field[\{0\} \cap \bigcup\deg(\field[x])]$. Also, $\localcoho{\maxid}{i}{\field[\affinesemigp]} =0$ except for $i=\dimd$, and $\localcoho{\maxid}{\dimd}{\field[\affinesemigp]} = \field[\regR_{\emptyset} \cap \bigcup\deg(\field[x])]$. Thus, the duality holds for this case.

For all other cases, $\Delta$ is a proper subcomplex of $\crosssection$.~\cref{thm:non_acyclic_chains_of_local_coho_of_maximal_id} shows that minimal open sets with index $(\facF,\facF)$ for all $\facF \in \bigcap\max(\Delta)$ may have nonzero local cohomology. Also, \cref{lem:acylic_complexes_of_local_coho_over_radical_ideal} shows that minimal open sets $\regR_{\facF^{c}}\cap\Z\affinesemigp$ for any $\facF \in  \bigcap\max(\Delta)$ may have nonzero local cohomology. Lastly,~\cref{cor:duality_of_local_cohomologies_2} provides the desired local cohomology. 
\end{proof}

\begin{figure*}[t!]
\centering
\begin{tabular}{|c|c|c|c|c|}
\hline
$d$ &$\subcpx$& $\ridI$ & $\deg_{\Z^{d}} H_{\maxid}^{i}(\field[\affinesemigp]/\ridI)$ & $\deg_{\Z^{d}} H_{\ridI}^{i}(\field[\affinesemigp])$\\
\hline
1 &1-sim& 0 & $(\emptyset, \regr_{\emptyset})$ & $(\regr_{1}, \emptyset)$\\
1 &$\emptyset$& $\maxid$ &$(\{ 0\},\emptyset)$ & $(\emptyset, \regr_{\emptyset}) $ \\
\hline
2 &2-sim& 0 &  $(\emptyset, \emptyset,\regr_{\emptyset})$ & $(\regr_{1,2}, \emptyset,\emptyset)$ \\
2 &pt& $\langle x\rangle$ & $(\emptyset, \regr_{2}, \emptyset  )$  & $(\emptyset, \regr_{1},\emptyset)$\\
2 &pt& $\langle y\rangle$ & $(\emptyset, \regr_{1}, \emptyset  )$  & $(\emptyset, \regr_{2},\emptyset)$\\
2 &2 pts& $\langle xy\rangle$ & $(\emptyset,\regr_{1}\cup\regr_{2} \cup\{0\}, \emptyset  )$  & $(\emptyset, \regr_{1}\cup \regr_{2} \cup \regr_{\emptyset},\emptyset)$\\
2 &$\emptyset$& $\maxid$ & $(\{0\}, \emptyset,\emptyset )$  & $(\emptyset, \emptyset,\regr_{\emptyset})$\\
\hline
3 &3-sim& 0 & $(\emptyset,\emptyset,\emptyset,\regr_{\emptyset} )$ &$(\regr_{1,2,3}, \emptyset,\emptyset,\emptyset)$\\
3 &2-sim(yz)& $\langle x\rangle$ &  $(\emptyset, \emptyset,\regr_{1},\emptyset )$ &  $(\emptyset, \regr_{2,3},\emptyset,\emptyset )$ \\
3 &2-sim(xz)& $\langle y\rangle$ &  $(\emptyset, \emptyset,\regr_{2},\emptyset )$ &  $(\emptyset, \regr_{1,3},\emptyset,\emptyset )$ \\
3 &2-sim(xy)& $\langle z\rangle$ &  $(\emptyset, \emptyset,\regr_{3},\emptyset )$ &  $(\emptyset, \regr_{1,2},\emptyset,\emptyset )$ \\
3 &(xz,yz)& $\langle xy\rangle$ &  $(\emptyset, \emptyset,\regr_{1}\cup\regr_{2}\cup\regr_{1,2},\emptyset )$ &  $(\emptyset, \regr_{2,3}\cup \regr_{1,3} \cup \regr_{3},\emptyset,\emptyset )$ \\
3 &(xy,yz)& $\langle xz\rangle$ &  $(\emptyset, \emptyset,\regr_{1}\cup\regr_{3}\cup\regr_{1,3},\emptyset )$ &  $(\emptyset, \regr_{2,3}\cup \regr_{1,2} \cup \regr_{2},\emptyset,\emptyset )$ \\
3 &(xy,xz)& $\langle yz\rangle$ &  $(\emptyset, \emptyset,\regr_{2}\cup\regr_{3}\cup\regr_{2,3},\emptyset )$ &  $(\emptyset, \regr_{1,3}\cup \regr_{1,2} \cup \regr_{1},\emptyset,\emptyset )$ \\
3 &(xy,yz,xz)& $\langle xyz\rangle$ &  $(\emptyset,\emptyset, (\bigcup_{\substack{i=\{1 \},\{ 2\},\{3\},\\ \{1,2\},\{1,3\},\{2,3\}}}\regr_{i})\cup \{0 \}),\emptyset )$ &  $(\emptyset, (\bigcup_{i \neq \{ 1,2,3\}}\regr_{i}),\emptyset,\emptyset )$ \\
3 &(x,yz)& $\langle xy,xz\rangle$ &  $(\emptyset, \regr_{2,3} \cup \{ 0\}, \regr_{1} ,\emptyset )$ &  $(\emptyset, \regr_{2,3},\regr_{1}\cup\regr_{\emptyset},\emptyset )$ \\
3 &(y,xz)& $\langle xy,yz\rangle$ &  $(\emptyset, \regr_{1,3} \cup \{ 0\}, \regr_{2} ,\emptyset )$ &  $(\emptyset, \regr_{1,3},\regr_{2}\cup\regr_{\emptyset},\emptyset )$ \\
3 &(z,xy)& $\langle xz,yz\rangle$ &  $(\emptyset, \regr_{1,2} \cup \{ 0\}, \regr_{3} ,\emptyset )$ &  $(\emptyset, \regr_{1,2},\regr_{3}\cup\regr_{\emptyset},\emptyset )$ \\
3 &(y,z)& $\langle x,yz\rangle$ &  $(\emptyset, \regr_{2,3} \cup\regr_{1,2} \cup \{ 0\}, \emptyset ,\emptyset )$ & $(\emptyset,\emptyset, \regr_{1} \cup \regr_{3}\cup\regr_{\emptyset}, \emptyset) $  \\
3 &(x,z)& $\langle y,xz\rangle$ &  $(\emptyset, \regr_{1,3} \cup\regr_{1,2} \cup \{ 0\}, \emptyset ,\emptyset )$ & $(\emptyset,\emptyset, \regr_{2} \cup \regr_{3}\cup\regr_{\emptyset}, \emptyset) $  \\
3 &(x,y)& $\langle z,yz\rangle$ &  $(\emptyset, \regr_{1,3} \cup\regr_{2,3} \cup \{ 0\}, \emptyset ,\emptyset )$ & $(\emptyset,\emptyset, \regr_{2} \cup \regr_{1}\cup\regr_{\emptyset}, \emptyset) $  \\
3 &(z)& $\langle x,y\rangle$ &  $(\emptyset, \regr_{1,2}, \emptyset ,\emptyset )$ & $(\emptyset,\emptyset, \regr_{3}, \emptyset) $  \\
3 &(x)& $\langle y,z\rangle$ &  $(\emptyset, \regr_{2,3}, \emptyset ,\emptyset )$ & $(\emptyset,\emptyset, \regr_{1}, \emptyset) $  \\
3 &(y)& $\langle x,z\rangle$ &  $(\emptyset, \regr_{1,3}, \emptyset ,\emptyset )$ & $(\emptyset,\emptyset, \regr_{2}, \emptyset) $  \\
3 &(x,y,z)& $\langle xy,xz,yz\rangle$ &  $(\emptyset, \regr_{2,3}\cup\regr_{1,3}\cup\regr_{1,2}\cup\{0\}^{2}, \emptyset ,\emptyset )$ & $(\emptyset,\emptyset,  \regr_{1}\cup\regr_{2}\cup\regr_{3}\cup\regr_{\emptyset}^{2}, \emptyset) $  \\
3& $\emptyset$ & $\maxid$ & $(\{ 0\}, \emptyset,\emptyset,\emptyset)$&$(\emptyset,\emptyset,\emptyset,\regr_{\emptyset} )$\\
\hline
\end{tabular}
\caption{Table of local cohomologies over a simplicial affine semigroup ring $\field[\affinesemigp]$ when whose dimension $d$ is 1,2, or 3.}
\label{fig:duality_simplicial_affinesemigp}
\end{figure*}
\begin{example}
\label{ex:low_dim_cases}
In~\cref{fig:duality_simplicial_affinesemigp}, we summarize the degrees of local cohomology over a simplicial affine semigroup ring $\field[\affinesemigp]$ (with $\dim\field[\affinesemigp]=1,2,$ or $3$) and with a radical monomial ideal $J_{\Delta}$ from a simplicial complex generated by $\{x\}$ (resp. $\{x,y\}$ or $\{x,y,z\}$) corresponding to the variables of $\field[\affinesemigp]$. 

For example, the 12th row of the table illustrates that, when $\field[Q]$ is a 3-dimensional simplicial affine semigroup ring, and $\Delta$ is a simplicial complex consisting of two edges $\overline{xy}$ and $\overline{yz}$, then its corresponding radical monomial ideal is $\<xy\>$, thus for any $\overline{x^{\vec{u}}} \in \field[\affinesemigp]/J_{\Delta}$ with $\vec{u} \in \regr_{1}\cup\regr_{2}\cup \regr_{1,2}$ and $x^{\vec{v}} \in \field[\affinesemigp]$ with $\vec{v} \in \regr_{2,3}\cup \regr_{1,3} \cup \regr_{3},$
\[
H_{\maxid}^{2}(\field[\affinesemigp]/\ridI)_{\vec{u}} \cong H_{J_{\Delta}}^{3-2}(\field[\affinesemigp]/\ridI)_{\vec{v}}
\]
and for all other $\vec{u}$ and $\vec{v}$,
\[
H_{\maxid}^{i}(\field[\affinesemigp]/\ridI)_{\vec{u}} = H_{J_{\Delta}}^{i}(\field[\affinesemigp]/\ridI)_{\vec{v}}=0.
\]
\end{example}

\bibliographystyle{plain}
\bibliography{GIshida}
\end{document}